\documentclass{kybernetika}

\usepackage{amsmath,amssymb,amsfonts}
\usepackage{dsfont}
\usepackage{graphicx}
\usepackage{subfigure}
\usepackage{algorithm,algorithmic}
\usepackage{color}
\usepackage{cite}
\usepackage{epstopdf}
\usepackage{graphicx}
\usepackage{bm}
\usepackage{mleftright}
\usepackage{mathrsfs}
\usepackage{booktabs}
\usepackage{hyperref}
\hypersetup
{
	colorlinks=true,
	linkcolor=blue,
	filecolor=blue,
	urlcolor=blue,
	citecolor=cyan,
}

\newtheorem{assumption}{Assumption}

\newtheorem{theorem}{Theorem}[section]
\newtheorem{lemma}[theorem]{Lemma}

\newtheorem{remark}[theorem]{Remark}

\newtheorem{proof}{Proof}
\begin{document}

\pagestyle{myheadings}

\title{Distributed Aggregative Optimization with Quantized Communication}

\author{Ziqin Chen and Shu Liang}

\contact{Ziqin}{Chen}{The Department of Control Science and Engineering, Tongji University, Shanghai, 201210, China}{cxq0915@tongji.edu.cn}
\contact{Shu}{Liang}{The Department of Control Science and Engineering, Tongji University, Shanghai, 201210, China}{sliang@tongji.edu.cn}

\markboth{Ziqin Chen and Shu Liang} {Distributed Aggregative Optimization with Quantized Communication}
\maketitle
\begin{abstract}
In this paper, we focus on an aggregative optimization problem under communication bottleneck. The aggregative optimization is to minimize the sum of local cost functions. Each cost function depends on not only local state variables but also the sum of functions of global state variables. The goal is to solve the aggregative optimization problem through distributed computation and local efficient communication over a network of agents without a central coordinator. Using the variable tracking method to seek the global state variables and the quantization scheme to reduce the communication cost spent in the optimization process, we develop a novel distributed quantized algorithm, called D-QAGT, to track the optimal variables with finite bits communication. Although quantization may lose transmitting information, our algorithm can still achive the exact optimal solution with linear convergence rate. Simulation experiments on an optimal placement problem is carried out to verify the correctness of the theoretical results.
\end{abstract}



\section{INTRODUCTION}
Distributed optimization has gained much research attention due to its wide applications in multi-agent network systems, such as resource allocation~\cite{yingyong1}, machine learning~\cite{yingyong21,yingyong2} and cloud computing~\cite{yingyong3}. In distributed optimization, each agent only uses local data and transmits information with its neighbors to minimize a global cost function cooperatively. To solve it, many efficient algorithms have been proposed, such as consensus-based algorithms~\cite{consensusbased1,consensusbased2,consensusbased3,consensusbased4,consensusbased5} and dual-decomposition-based algorithms~\cite{dualdecompose1,dualdecompose2,dualdecompose3}.

However, when the number of agents grows, neighboring communication is much slower than computation, which induces a bottleneck to run the above distributed optimization algorithm~\cite{communicationefficient}. 
Therefore, developing communication-efficient algorithms becomes a new research hotspot. 
Quantization techniques aim to compress information by reducing the number of bits per communication, and have been successfully applied to several large-scale engineering tasks recently~\cite{quantizationyingyong}. With regards to distributed optimization problems, the static quantization method has been utilized in~\cite{static1,static2}, which find the optimal solution with some fixed error bound. For removing the quantization errors,  the 3-bit dynamic quantization method with an adjustable quantization level was developed in~\cite{yipeng}. 
It enables distributed quantized subgradient algorithms to achieve exact optimal solutions. 
After that, various distributed quantized optimization algorithms have arisen, including distributed alternating direction method of multipliers with dynamic quantization methods~\cite{ADMMQ}, distributed quantized gradient tracking algorithms~\cite{GtrackingQ} and distributed subgradient descent algorithms with amplified-differential compression methods~\cite{CompressQ}.

It should be noted that the local cost function studied in the aforementioned works only depends on its own state variable, like the form as $f_{i}(x_{i})$ associated with inequality constraints $x_{i}=x_{j},~i\neq j$. However, in many practical applications, such as multi-agent formation control, optimal placement problem and transportation networks and formation control, local cost functions $f_{i}$ are not only determined by its own variable $x_{i}$ but also influenced by any other agents' variables $x_{j}$. Hence, a novel framework for distributed optimization, called aggregative optimization, was investigated in~\cite{aggregative1,aggregative2,lixiuxian}. The form of aggregative optimization is often described as follows, 
\begin{eqnarray}
	\begin{aligned}
		\min _{\boldsymbol{x} \in \mathbb{R}^{Nn}} f(\boldsymbol{x}) &:=\sum_{i=1}^{N} f_{i}\left(x_{i}, \chi(\boldsymbol{x})\right), \\
		\chi(\boldsymbol{x}) &:=\frac{\sum_{i=1}^{N} g_{i}\left(x_{i}\right)}{N},\label{primalproblem}
	\end{aligned}
\end{eqnarray}
where $\boldsymbol{x}=[x_{i}]_{i\in\{1,\cdots,N\}}$ is the global state variable with the local state variable $x_{i}\in{\mathbb{R}^{n}}$. The term $\chi(\boldsymbol{x}): \mathbb{R}^{Nn}\rightarrow \mathbb{R}^{r}$ is an aggregative variables associated with all agent's state variables. In the distributed aggregative optimization, the local cost function $f_{i}\left(x_{i}, \chi(\boldsymbol{x})\right): \mathbb{R}^{Nn}\rightarrow \mathbb{R}$ and $g_{i}(x_{i}): \mathbb{R}^{n}\rightarrow \mathbb{R}^{r}$ are only privately known by agent $i$. 

For solving the problem~\eqref{primalproblem}, the related work~\cite{lixiuxian} proposes the distributed aggregative gradient tracking (D-AGT) algorithm with infinity precision communication, however, it is not suitable for considering the communication bottleneck. 
So far as we know, distributed quantized algorithm for solving aggregative optimization in~\eqref{primalproblem} has not been proposed. 
The main difficulties of distributed aggregative optimization with quantization may lie in the following three points: i) Each local cost function $f_{i}$ depends on aggregative variables $\chi(\boldsymbol{x})$, while the gobal aggregative variables and their gradient can not be access to any individual agent $i$. ii) Only finite bits quantized variables are transmitted between neighboring agents, and it further provides less available information for each agent $i$. iii) Quantization introduces noise to algorithm updates and thus deteriorates convergence in general. Hence, designing a distributed quantized algorithm for solving aggregative optimization with communication bottleneck is challenging and meaningful.  Motivated by the above facts, the contributions of this paper are summarized as follows.
\begin{itemize}
	\item This paper studies distributed aggregative optimization problem \eqref{primalproblem}. Particularly, the local cost function of each agent depends both on its own variable and the aggregative variable, which is the global information that any individual agent can not know. This problem can cover more transitional optimization problems in~\cite{yingyong1,yingyong2,yingyong3,consensusbased1,consensusbased2,consensusbased3,consensusbased4}. Meanwhile, the communication bottleneck is also considered. The only finite number of bits are allowed to interact between the neighboring agents, making it possible to employ the practical bandwidth constraints situation rather than just the infinity communication~\cite{aggregative1,aggregative2,lixiuxian}.
		
	\item Based on the classical gradient descent algorithm and the quantization technique, we propose a distributed quantized aggregative gradient tracking (D-QAGT) algorithm for solving the problem \eqref{primalproblem}. The novel algorithm uses the variable tracking method to estimate the global aggregative term and its corresponding gradient such that it can converge to the optimal solution at the linear convergence rate. Furthermore, the proposed D-QAGT is a communication-efficient algorithm that significantly saving in the communication overhead in transmitted bits.
	
	\item A practical optimal placement problem is used in simulation to demonstrate the convergence of the D-QAGT algorithm. Numerical tests show that under $5$-bits transmission in each communication step, the proposed D-QAGT algorithm could keep the same linear convergence with the D-AGT algorithm in~\cite{lixiuxian}, which utilizes infinity precision communication.
\end{itemize}	

This paper is organized as follows. Section 2 introduces some related preliminaries
on basic notations, graph theory and formulates the distributed aggregative optimization problem via two examples.
Second 3 provides the proposed distributed algorithm and analyzes its convergence performance. Then, Section 4 gives a numerical experiment and Section 5 concludes the paper.

\section{PRELIMINARIES AND PROBLEM FORMULATION}
\subsection{Basic notations and notions}
Denote by $\mathbb{R}^{n}$ the set of real vectors with $n$-dimension, $\mathbb{R}^{n}_{+}$ the set of vectors with nonnegative coordinates of $n$-dimension and $\mathbb{R}^{n\times m}$ the set of real matrices with $n$-rows and $m$-columns. Denote by $\mathbb{N}$ an integer set. Let $\|\cdot\|$ and $\otimes$ be the standard Euclidean norm and the Kronecker product, respectively. Denote by $\mathbf{1}_{N}$ and $\mathbf{0}_{N}$ the column vectors of $N$ dimension with all entries being $1$ and $0$, respectively. Let $I_{n}$ be the compatible identity matrix with dimension $n$. For any vector $x,y\in{\mathbb{R}^{n}}$, let $x^{T}$ be the transpose of $x$. Denote by $[x_{i}]_{i\in \Omega}$ the column vector by stacking up $x_{i}$ associated with $i\in \Omega$. For any square matrix $H$, denote by $\rho(H)$ the spectral radius. 

A differentiable function $f(x): \mathbb{R}^{n}\rightarrow\mathbb{R}$ is called $\mu$-strongly convex for $\mu>0$ if for any $x,y\in \mathbb{R}^{n}$, $\langle \nabla_{x}f(x)-\nabla_{y}f(y),x-y\rangle\geq \mu\|x-y\|^2$. The gradient of the function $f(x)$ is Lipschitz continuous with constant $m>0$ if for any $x,y\in \mathbb{R}^{n}$, $\|\nabla_{x}f(x)-\nabla_{y}f(y)\|\leq m\|x-y\|$.

\subsection{Graph theory}
A directed and strongly connected graph of a multi-agent system is denoted by $\mathcal{G}=(\mathcal{E},\mathcal{V})$, where the node set $\mathcal{V}=\{1,\cdots,N\}$ and the edge set $\mathcal{E}\subseteq \mathcal{V}\times \mathcal{V}$. $(j,i)\in\mathcal{E}$ with $i\neq j$ means that agent $j$ can sent information to $i$. Denote $\mathcal{N}_{i}=\{j\in{\mathcal{V}}:(j,i)\in{\mathcal{E}}\}$ is the neighbor set of agent $i$. The adjacency matrix $A=(a_{ij})\in{R^{N\times N}}$ is defined by $a_{ij}>0$ if $(j,i)\in\mathcal{E}$, and $a_{ij}=0$ otherwise. Let $d_{i}=|N_{i}|$ denote the degree of node $i$ and $D$ be the $N\times N$ diagonal matrix such that $D_{ii}=d_{i}$. Then, the Laplacian matrix is denoted by $L=D-A$. It is necessary to display the following assumption and lemma. 

\begin{assumption}\label{assumption1}
The graph $\mathcal{G}$ is directed and strongly connected, and the adjacency matrix $A\in{\mathbb{R}^{N\times N}}$ is doubly stochastic, i.e, $\mathbf{1}_{N}^{T}A=\textbf{1}_{N}^{T}$ and $A\textbf{1}_{N}=\textbf{1}_{N}$.
\end{assumption}

\begin{lemma}{\cite{Lemma1}}\label{Lemma1}
Under Assumption~\ref{assumption1}, the following properties hold for the adjacency matrix.

i) $(A\otimes I_{n})(\frac{1}{N}\textbf{1}_{N}\textbf{1}_{N}^{T}\otimes I_{n})=(\frac{1}{N}\textbf{1}_{N}\textbf{1}_{N}^{T}\otimes I_{n})(A\otimes I_{n})=\frac{1}{N}\textbf{1}_{N}\textbf{1}_{N}^{T}\otimes I_{n}$.

ii) $\|A-\frac{1}{N}\textbf{1}_{N}\textbf{1}_{N}^{T}\|<1$.

iii) $\|A-I_{N}\|\leq 2.$

iv) $\|(A\otimes I_{n})x-(\frac{1}{N}\textbf{1}_{N}\textbf{1}_{N}^{T}\otimes I_{n})x\|\leq \|A-\frac{1}{N}\textbf{1}_{N}\textbf{1}_{N}^{T}\|\|x-(\frac{1}{N}\textbf{1}_{N}\textbf{1}_{N}^{T}\otimes I_{n})x\|$.
\end{lemma}

\subsection{Problem formulations}

The goal of the agents is to collectively seek an optimal solution for aggregative optimization problem~\eqref{primalproblem}, whose global function is the sum of the local cost functions. It is different from the aggregative game~\cite{aggregativegame1,aggregativegame2}, where each agent aims to minimize its own local cost function. In most situations, the global performance of \eqref{primalproblem} may be better than that of the most related aggregative game, which is illustrated by the following example.

\textbf{Example 1:} (Bandwidth sharing problem) Consider $N$ users each wants to have part of a shared resource. Assume that the maximum capacity of a shared channel says $1$ and the strategy of each user $i,~i\in\{\mathcal{V}\}$ is to send $x_{i}$ units of flow along the channel for some value $x_{i}\in[0,1]$. If the total bandwidth $\sum_{i=1}^{N}x_{i}\geq1$, no user gets any benefit. If $\sum_{i=1}^{N}x_{i}<1$, then the local payoff function of each user $i$ is $f_{i}(x_{i}, \chi(\boldsymbol{x}))=x_{i}(1-\chi(\boldsymbol{x}))$ with $\chi(\boldsymbol{x})=\sum_{j\in{N}}x_{j}$. Based on~\cite{example2}, when each user aims to maximize its payoff, the optimal strategies for user $i,~i\in\{1,\cdots,N\}$ is $x_{i}^*=1/(N+1)$. Thus, the Nash equilibrium is $(1/(N+1),\cdots,1/(N+1))$. In this aggregative game, the local payoff of each user $i$ is $f_{i}=(1/(N+1))^2$ and the global payoff is $\sum_{i=1}^{N}f_{i}=N/(N+1)^2$.

However, If all agents cooperate to maximize the global payoff function, the optimal solution is computed as $\boldsymbol{x}^*=(1/2N,\cdots,1/2N)$. In this setting, the local payoff of each user $i$ is $f_{i}=1/4N$ and the global payoff is $\sum_{i=1}^{N}f_{i}=1/4>N/(N+1)^2$. This example indicates that all agents will perform better in a cooperative manner compared with the aggregative game in a noncooperative manner, which motivates us to study the aggregative optimization as~\eqref{primalproblem}.

Note that in the cooperation process among agents, neighboring communication by assuming infinite precision often incurs expensive communication bandwidth. In order to solve \eqref{primalproblem} with the cost of all communications as low as possible, we adopt quantized communication. The main idea of quantization is to map the input data to a countable set of code values. More specifically, we use a simple yet effective uniform quantizer to divide the input domain into the same size quantization cells, in which a unique code value represents all elements within a cell. For a pair of $(j,i)\in{\mathcal{E}}$, agent $j$ transmits this code value to its neighbors $i$, and the agents $i$ recover the state of agent $j$ based on received code value. To illustrate the problem \eqref{primalproblem} with quantized communication better, we present the following practical example.

\textbf{Example 2:} (Optimal placement problem) Consider $N$ agents to protect a target at position $p_{0}\in{\mathbb{R}_{+}^{2}}$. We hope that some weighted center of all agents could track the target $p_{0}$. Denote the position of agent $i$ by $x_{i}\in{\mathbb{R}_{+}^{2}}$ and the weight center of all agents by $\chi(\boldsymbol{x})=\sum_{i=1}^{N}\sqrt{Nx_{i}^2}/N$. Then the problem is cast as problem \eqref{primalproblem} with $\sum_{i=1}^{N}f_{i}(x_{i},\chi(\boldsymbol{x}))=\sum_{i=1}^{N}\alpha_{i}\|\chi(\boldsymbol{x})-p_{0}\|^2$, where $\alpha_{i}>0$ is the weight constant.

In the distributed framework, $\chi(\boldsymbol{x})$ is the global information and cannot be known directly for all agents. Hence, $\chi_{i}\in{\mathbb{R}^{2}},~i\in\{1,\cdots,N\}$ is leveraged for each agent $i$ to track the $\chi(\boldsymbol{x})$, and $\chi_{i}$ need to be interacted between neighbors over network graph $\mathcal{G}$. However, transmitting $\chi_{i}$ would consume amount of bandwidth resource, so we quantize $\chi_{i}$ before transmitting it. Particularly, by using quantizer $Q(\cdot)$, each coordinate of real value $\chi_{i}$ is mapped into a countable set of code values. i.e., $Q(\chi_{i}): \mathbb{R}^{2}_{+}\rightarrow \mathbb{N}_{+}^2$.  Then agent $i$ just transmits the code of quantized message $Q(\chi_{i})$ rather than that of the raw vectors $\chi_{i}$ such that the reduction of communication cost. More details on the design of the quantization scheme will be given subsequently.

To move forward, define $f(\boldsymbol{x},\boldsymbol{z})=\sum_{i=1}^{N}f_{i}(x_{i},z_{i}):\mathbb{R}^{Nn+Nr}\rightarrow \mathbb{R}$ for any $\boldsymbol{x}=[x_{i}]_{i\in\mathcal{V}}\in{\mathbb{R}^{Nn}}$ and $\boldsymbol{z}=[z_{i}]_{i\in\mathcal{V}}\in{\mathbb{R}^{Nr}}$. The gradients of $f(\boldsymbol{x},\boldsymbol{z})$ is defined by $\nabla_{\boldsymbol{x}}f(\boldsymbol{x},\boldsymbol{z})=[\nabla_{x_{i}}f_{i}(x_{i},z_{i})]_{i\in{\mathcal{V}}}$ and $\nabla_{\boldsymbol{z}}f(\boldsymbol{x},\boldsymbol{z})=[\nabla_{z_{i}}f_{i}(x_{i},z_{i})]_{i\in{\mathcal{V}}}$, respectively. We make the following necessary assumption.

 \begin{assumption}\label{assumption2}
 The functions arisen in \eqref{primalproblem} satisfy
 
 i) The cost function $f$ is differentiable, $\mu$-strongly convex and locally $l_{1}$-smooth.
 
 ii) $\nabla_{\boldsymbol{z}}f(\boldsymbol{x},\boldsymbol{z})$ is locally $l_{2}$-Lipschitz continuous.
 
 iii) For all $i\in{\mathcal{V}}$, $g_{i}(x_{i})$ are differentiable and $\nabla_{x_{i}}g_{i}(x_{i})\in{R^{n\times r}}$ is locally bounded, i.e., $\|\nabla_{x_{i}}g_{i}(x_{i})\|\leq \frac{l_{3}}{N}$ for some positive $l_{3}$.
\end{assumption}
Note that Property i) means that $\nabla f$ and
\begin{eqnarray}
\nabla_{\boldsymbol{x}}f(\boldsymbol{x},\boldsymbol{z})+\nabla_{\boldsymbol{x}}g(\boldsymbol{x})\mathbf{1}_{N}\otimes \frac{1}{N}\sum_{i=1}^{N}\nabla_{z_{i}}f_{i}(x_{i},z_{i}),\label{as}
\end{eqnarray}
are locally $l_{1}$-Lipschitz continuous with $g(\boldsymbol{x})=[g_{i}(x_{i})]_{i\in{\mathcal{V}}}: \mathbb{R}^{Nn}\rightarrow \mathbb{R}^{Nr}$. At the same time, Property iii) ensures that $\|\nabla_{\boldsymbol{x}}g(\boldsymbol{x})\|\leq l_{3}$.

\section{Main Result}
In this section, we present the design of distributed quantized algorithm and give the convergence analysis.
\subsection{Distributed quantized algorithm}

For solving the problem~\eqref{primalproblem}, each agent $j,~j\in{\mathcal{V}}$ holds two states $\chi_{j}\in{\mathbb{R}^{r}}$ and $y_{j}\in{\mathbb{R}^{r}}$ to track the average $\chi(\boldsymbol{x})$ and the gradient sum $\frac{1}{N}\sum_{j=1}^{N}\nabla_{\chi}f_{j}(x_{j},\chi(\boldsymbol{x}))$, respectively. Meanwhile, each agent quantizes $\chi_{j}$ and $y_{j}$ for the information interaction. Thus, each agent is associated with an encoder and its neighbors possess the corresponding decoder. Firstly, the uniform quantizer $q(\cdot)$ is introduced as follows.

\textbf{Quantizer:} A uniform quantizer is described by the function $q(\cdot): \mathbb{R}\rightarrow \mathbb{N}$, in which
\begin{eqnarray}
q(x)= \begin{cases}0, & \text { if }-\frac{1}{2} \leq x \leq \frac{1}{2}, \\ i, & \text { if } \frac{2 i-1}{2}<x \leq \frac{2 i+1}{2}, i=1, \cdots, L, \\ L, & \text { if } x>\frac{2 L+1}{2}, \\ -q(-x), & \text { if } x<-\frac{1}{2} .\end{cases}\label{quantizer}
\end{eqnarray}
In practice, it is not necessary to transmit any information when the output of the quantizer is zero, thus, the communication process of each agent is required to transmit $\lceil\log_{2}(2L)\rceil$ for the above $2L+1$-level quantizer \eqref{quantizer}. For clarify, define $Q(\mathbf{a})$ for the vector $\mathbf{a}=[a_{i}]_{i\in\{1,\cdots,N\}}\in \mathbb{R}^{N}$ by $Q(\mathbf{a})=[q(a_{i})]_{i\in\{1,\cdots,N\}}\in \mathbb{R}^{N}$.

Next, we design an encoder-decoder scheme for each pair of agents $(j,i),~j\in{\mathcal{N}_{i}}$. Agent $j$ quantizes its states $\chi_{j}$ and $y_{j}$, then transmits the code of $Q(\chi_{j})$ and $Q(y_{j})$ to its neighbors $i$. Agent $i$ receives the code of $Q(\chi_{j})$ and $Q(y_{j})$ and estimates the states of agent $j$ denoted by $\hat{\chi}_{j}$ and $\hat{y}_{j},~j\in{\mathcal{N}_{i}}$.

\begin{table}[H]
	\renewcommand\arraystretch{1.2}
	\label{dsmb}
	\centering
	\small
	\begin{tabular}{p{0.95\linewidth}}
		\toprule[1.5pt]
		{\textbf{Encoder:} The encoder is installed in agent $j,~j\in{\mathcal{V}}$.} \\
		\midrule[1pt]
	Agent $j$ generates $r$-dimensional quantized outputs $s_{\chi j}(k)$ and $s_{y j}(k)$, and then transmits them to neighbors $i$ for any $k\in{\mathbb{N}}$,
	\begin{eqnarray}
		&&s_{\chi j}(0)=Q\left(\frac{\chi_{j}(0)}{l(0)}\right),~~~~~s_{yj}(0)=Q\left(\frac{y_{j}(0)}{l(0)}\right),~~~~~l(k)=l(0)\gamma^{k},\\
			&&s_{\chi{j}}(k+1) :=Q\left(\dfrac{\chi_{j}(k+1)-\hat{\chi}_{j}(k)}{l(k+1)}\right),\label{ecdchi}\\
		&&s_{y j}(k+1) :=Q\left(\dfrac{y_{j}(k+1)-\hat{y}_{j}(k)}{l(k+1)}\right),\label{ecdy}
	\end{eqnarray}
	where $l(k)$ is the decaying scaling function with any positive initial constant $l(0)>0$. The rate of the scaling function $\gamma$ is given in the following~\eqref{gamma}. 
	\\
	\bottomrule[1.5pt]
	\end{tabular}
\end{table}

Agent $i$ receives quantized outputs $s_{\chi j}$ and $s_{y j}$ from its neighbors $j$, and then estimates its neighbors's states through a decoder defined as follows.

\begin{table}[H]
	\renewcommand\arraystretch{1.2}
	\label{dsmb}
	\centering
	\small
	\begin{tabular}{p{0.95\linewidth}}
		\toprule[1.5pt]
		{\textbf{Decoder:} The decoder is installed in agent $i,~j\in{\mathcal{N}_{i}}$.} \\
		\midrule[1pt]
		Agent $i$ receives $r$-dimensional quantized outputs $s_{\chi j}(k)$ and $s_{y j}(k)$ from its neighbors $j,~j\in{\mathcal{N}}_{i}$, and then recovers its neighbors' states $\hat{\chi}_{j}$ and $\hat{y}_{j}$ as follows,
		\begin{eqnarray}
			&&\hat{\chi}_{j}(0)=l(0)s_{\chi{j}}(0),~~~~~\hat{y}_{j}(0)=l(0)s_{y{j}}(0),\\
			&&\hat{\chi}_{j}(k+1)=l(k+1)s_{\chi j}(k+1)+\hat{\chi}_{j}(k),\label{dcdchi}\\
			&&\hat{y}_{j}(k+1)=l(k+1)s_{y j}(k+1)+\hat{y}_{j}(k),~k\in{\mathbb{N}},\label{dcdy}
		\end{eqnarray}
	where the design of $l(k)$ is the same as the encoder.\\
		\bottomrule[1.5pt]
	\end{tabular}
\end{table}
\begin{remark}
	It should be noted that all the agents possess the same scaling function $l(k)>0$. In dynamic quantized control, the scaling function represents the quantized precision and is designed as a decaying sequence to adaptively adjust the encoder. For convergence analysis, the scaling function must be designed carefully such that the agents gradually increase the accuracy of states recovery of its neighbors and the quantization error gradually decays to zero. 
\end{remark}

Based on the quantized communication associated with the above encoder-decoder pair, for $i \in{\mathcal{V}}$, the $i$th agent updates its real-valued state $x_{i}\in{\mathbb{R}^{n}}$ via the following distributed quantized algorithm.

\begin{table}[H]\label{A1}
	\renewcommand\arraystretch{1.2}
	\label{Algroithm1}
	\centering
	\small
	\begin{tabular}{p{0.95\linewidth}}
		\toprule[1.5pt]
		{\textbf{Algorithm 1:} Distributed Quantized Aggregative Gradient Tracking (D-QAGT)} \\
		\midrule[1pt]
		For any $k\in{\mathbb{N}}$, each agent $i,~i\in{\mathcal{V}}$ updates its states $x_{i}$, $\chi_{i}$ and $y_{i}$ as follows,
\begin{eqnarray}
			x_{i}(k+1) \!\!\!\!\!&=&\!\!\!\!\!\! x_{i}(k)-\alpha\left[\nabla_{x_{i}}f_{i}(x_{i}(k),\chi_{i}(k))+\nabla_{x_{i}} g_{i}(x_{i}(k))y_{i}(k)\right],\label{x}\\
			\chi_{i}(k+1) \!\!\!\!\!\!&=&\!\!\!\!\!\!\!\sum_{j=1}^{N}a_{ij}\hat{\chi}_{j}(k)+g_{i}(x_{i}(k+1))-g_{i}(x_{i}(k))+\chi_{i}(k)-\hat{\chi}_{i}(k),\label{chi}\\
			y_{i}(k+1) \!\!\!\!\!\!&=&\!\!\!\!\!\!\!\sum_{j=1}^{N}a_{ij}\hat{y}_{j}(k)\!\!+\!\!\nabla_{\chi_{i}}\!f_{i}(x_{i}(k\!+\!1),\chi_{i}(k\!\!+\!1)\!)\!\!-\!\!\nabla_{\chi_{i}}\!f_{i}(x_{i}(k),\chi_{i}(k))\!\!+\!y_{i}(k)\!\!-\!\hat{y}_{i}(k),\label{y}
		\end{eqnarray}
		where the stepsize $\alpha$ is given in the following~\eqref{alpha}. \\
		\bottomrule[1.5pt]
	\end{tabular}
\end{table}
The parameters in Algorithm 1 and Encoder-Decoder scheme are chosen as follows.
\begin{itemize}
\item The initial states $x_{i}(0)$, $\chi_{i}(0)$ and $y_{i}(0)$ satisfy

    i) $\|x_{i}(0)-x^*\|_{\infty}\leq c_{0}$.\\
    ii) $\chi_{i}(0)=g_{i}(x_{i}(0))$ and $\|\chi_{i}(0)\|_{\infty}\leq c_{1}$.\\
    iii) $y_{i}(0)=\nabla_{\chi_{i}}f_{i}(x_{i}(0),\chi_{i}(0))$ and $\|y_{i}(0)\|_{\infty}\leq c_{2}$.
\item The stepsize $\alpha$ satisfies
\begin{equation}
\alpha\in\left(0,\frac{\mu(1-\kappa)^2}{l_{3}(\mu+l_{1}+l_{2}l_{3})((1-\kappa)(l_{1}+l_{2}+l_{2}l_{3})+2l_{2}l_{3})}\right),\label{alpha}
\end{equation}
with the constant $\kappa=\|A-\frac{1}{N}\textbf{1}_{N}\textbf{1}_{N}^{T}\|<1$.
\item The rate of the scaling function $\gamma$ satisfies
\begin{equation}
	\gamma\in\left(\rho(H),1\right), \label{gamma}
\end{equation}
where $\rho(H)$ is the spectral radius of $H(\alpha)$ defined as
\begin{equation}
	H(\alpha)=\left[
	\begin{array}{ccc}
		(1-\mu\alpha) & \alpha l_{1} & \alpha l_{3}  \\
		\alpha l_{1}l_{3}(1+l_{3}) & \kappa+\alpha l_{1} l_{3} & \alpha l_{3}^2 \\
		\alpha l_{1}l_{2}(1+l_{3})^2 & \alpha l_{1}l_{2}(1+l_{3})+2l_{2} & \kappa+\alpha l_{2}l_{3}(1+l_{3}) \\
	\end{array}
	\right].\label{H}
\end{equation}
\end{itemize}

Noticing that D-QAGT algorithm relies on the quantized communication via using the estimated states $\hat{\chi}_{j}$ and $\hat{y}_{j},~j\in{\mathcal{N}}_{i}$. It is clear that D-QAGT algorithm merely requires discrete-time communication with finite bits at each round of communication, and it can resolve a finite bandwidth bottleneck. Regarding this fact, the D-QAGT saves the communication resource and broadens the range of application of D-AGT in~\cite{lixiuxian}.
 
The Q-DAGT algorithm combines the classical gradient descent algorithm, the variables tracking techniques and quantization communication methods. It is the first proposed to solve the aggregative network optimization problem with the communication bottleneck.
 
\subsection{Convergence analysis}

In this section, we present the convergence analysis for the D-QAGT algorithm. We first formulate a compact form of dynamics~\eqref{x}-\eqref{y} and then find a fixed point, which is also an optimal solution to problem~\eqref{primalproblem}. For explicit illustration, we introduce the following notations:
\begin{flalign}
		\boldsymbol{x}&=[x_{i}]_{i\in{\mathcal{V}}}\in{\mathbb{R}^{Nn}},~~~~~~~~~~~~~~~~~~~~~~~~~~~~~~~~~\boldsymbol{L}=(L\otimes I_{r})\in{\mathbb{R}^{Nr\times Nr}},\nonumber\\ \boldsymbol{\chi}&=[\chi_{i}]_{i\in{\mathcal{V}}}\in{\mathbb{R}^{Nr}},~~~~~~~~~~~~~~~~~~~~~~~~~~~~~~~~~ \boldsymbol{y}=[y_{i}]_{i\in{\mathcal{V}}}\in{\mathbb{R}^{Nr}}, \nonumber \\
		\bar{\chi}&=\frac{1}{N}\sum_{i=1}^{N}\chi_{i}\in{\mathbb{R}^{r}},~~~~~~~~~~~~~~~~~~~~~~~~~~~~~~~~~\bar{y}=\frac{1}{N}\sum_{i=1}^{N}y_{i}\in{\mathbb{R}^{r}},\nonumber\\
		\hat{\boldsymbol{\chi}}&=[\hat{\chi}_{i}]_{i\in{\mathcal{V}}}\in{\mathbb{R}^{Nr}},~~~~~~~~~~~~~~~~~~~~~~~~~~~~~~~~~\hat{\boldsymbol{y}}=[\hat{y}_{i}]_{i\in{\mathcal{V}}}\in{\mathbb{R}^{Nr}},\nonumber\\
		\boldsymbol{e}_{\boldsymbol{\chi}}&=\boldsymbol{\chi}-\hat{\boldsymbol{\chi}}\in{\mathbb{R}^{Nr}},~~~~~~~~~~~~~~~~~~~~~~~~~~~~~~~~\boldsymbol{e}_{\boldsymbol{y}}=\boldsymbol{y}-\hat{\boldsymbol{y}}\in{\mathbb{R}^{Nr}},\nonumber\\
		g(\boldsymbol{x})&=[g_{i}(x_{i})]_{i\in{\mathcal{V}}}\in{\mathbb{R}^{Nr}},~~~~~~~~~~~~~~~~~~~~\nabla_{\boldsymbol{x}} g(\boldsymbol{x})=[\nabla_{x_{i}} g_{i}(x_{i})]_{i\in{\mathcal{V}}}\in{\mathbb{R}^{Nn\times Nr}},\nonumber\\
		\nabla_{\boldsymbol{x}}f(\boldsymbol{x},\boldsymbol{\chi})&=[\nabla_{x_{i}}f_{i}(x_{i},\chi_{i})]_{i\in{\mathcal{V}}}\in{\mathbb{R}^{Nn}},~~~~~~
		\nabla_{\boldsymbol{\chi}}f(\boldsymbol{x},\boldsymbol{\chi})=[\nabla_{\chi_{i}}f_{i}(x_{i},\chi_{i})]_{i\in{\mathcal{V}}}\in{\mathbb{R}^{Nr}}.\nonumber
\end{flalign}
Then by using $\boldsymbol{L}=I_{Nr}-(A\otimes I_{r})$, the iterations~\eqref{x}-\eqref{y} in D-QAGT algorithm can be written as the following compact form
\begin{eqnarray}
	\boldsymbol{x}(k+1) \!\!\!\!&=&\!\!\!\! \boldsymbol{x}(k)-\alpha\left[\nabla_{\boldsymbol{x}}f(\boldsymbol{x}(k),\boldsymbol{\chi}(k))+\nabla_{\boldsymbol{x}} g(\boldsymbol{x}(k))\boldsymbol{y}(k)\right],\label{xx}\\
	\boldsymbol{\chi}(k+1)\!\!\!\!\!&=&\!\!\!\!\!\boldsymbol{L}\boldsymbol{e}_{\boldsymbol{\chi}}(k)+g(\boldsymbol{x}(k+1))-g(\boldsymbol{x}(k))+(A\otimes I_{r})\boldsymbol{\chi}(k),\label{chie}\\
	\boldsymbol{y}(k+1) \!\!\!\!\!&=&\!\!\!\!\!\boldsymbol{L}\boldsymbol{e}_{\boldsymbol{y}}(k)\!\!+\!\nabla_{\boldsymbol{\chi}}f(\boldsymbol{x}(k\!+\!1),\boldsymbol{\chi}(k\!\!+\!1))\!\!-\!\nabla_{\boldsymbol{\chi}}f(\boldsymbol{x}(k),\boldsymbol{\chi}(k))\!+\!(A\otimes I_{r})\boldsymbol{y}(k).\label{ye}
\end{eqnarray}

Next, we establish the equivalence of the fixed point of the D-QAGT algorithm and the optimal solution to the problem \eqref{primalproblem}.

\begin{lemma}\label{Lemma2}
Denote $\boldsymbol{x}^*$, $\boldsymbol{\chi}^*$ and $\boldsymbol{y}^*$ as the fixed points of \eqref{xx}-\eqref{ye}. Under Assumptions 1 and 2, $\boldsymbol{x}^*$ is the optimal solution to problem \eqref{primalproblem}.
\end{lemma}
\begin{proof}
Based on Assumption~\ref{assumption1} on $\mathbf{1}_{Nr}^{T}\boldsymbol{L}=\mathbf{0}_{N}^{T}$, we multiply $\frac{1}{N}\mathbf{1}_{Nr}^{T}$ on both side of \eqref{chie} and \eqref{ye} to obtain that
\begin{eqnarray}
	\bar{\chi}(k+1)\!\!\!&=&\!\!\!\bar{\chi}(k)+\frac{1}{N}\sum_{i=1}^{N}g_{i}(x_{i}(k+1))-\frac{1}{N}\sum_{i=1}^{N}g_{i}(x_{i}(k)),\label{barchi}\\
	\bar{y}(k+1) \!\!\!&=&\!\!\!\!\bar{y}(k)\!+\!\frac{1}{N}\sum_{i=1}^{N}\nabla_{\chi_{i}}f_{i}(x_{i}(k\!+\!1),\chi_{i}(k\!\!+\!1))\!\!-\!\frac{1}{N}\sum_{i=1}^{N}\nabla_{\chi_{i}}f_{i}(x_{i}(k),\chi_{i}(k)).
\end{eqnarray}
It follows from a simple recursion that
\begin{eqnarray}
	\bar{\chi}(k)\!-\!\frac{1}{N}\sum_{i=1}^{N}g_{i}(x_{i}(k))\!\!\!&=&\!\!\!\bar{\chi}(0)\!-\!\frac{1}{N}\sum_{i=1}^{N}g_{i}(x_{i}(0)),\label{barchi2}\\
	\bar{y}(k)\!-\!\frac{1}{N}\sum_{i=1}^{N}\nabla_{\chi_{i}}f_{i}(x_{i}(k),\chi_{i}(k))\! \!\!\!&=&\!\!\!\!\bar{y}(0)\!-\!\frac{1}{N}\sum_{i=1}^{N}\nabla_{\chi_{i}}f_{i}(x_{i}(0),\chi_{i}(0)).\label{bary2}
\end{eqnarray}
Combing with $\chi_{i}(0)=g_{i}(x_{i}(0))$ and $y_{i}(0)=\nabla_{\chi_{i}}f_{i}(x_{i}(0),\chi_{i}(0))$,
 \begin{eqnarray}
 	\bar{\chi}(k)\!&=&\!\frac{1}{N}\sum_{i=1}^{N}g_{i}(x_{i}(k))=\chi(\boldsymbol{x}),\label{barchi3}\\
 	\bar{y}(k)&=&\frac{1}{N}\sum_{i=1}^{N}\nabla_{\chi_{i}}f_{i}(x_{i}(k),\chi_{i}(k)).\label{bary3}
 \end{eqnarray}
By substituting the fixed point $\boldsymbol{x}^*$, $\boldsymbol{\chi}^*$ and $\boldsymbol{y}^*$ into \eqref{xx}-\eqref{ye},
\begin{eqnarray}
	\nabla_{\boldsymbol{x}}f(\boldsymbol{x}^*,\boldsymbol{\chi}^*)+\nabla_{\boldsymbol{x}} g(\boldsymbol{x}^*)\boldsymbol{y}^*=\mathbf{0}_{Nn},\label{x*}\\
	\boldsymbol{L}\boldsymbol{\chi}^*=\mathbf{0}_{Nr},~~\boldsymbol{L}\boldsymbol{y}^*=\mathbf{0}_{Nr},\label{y*}
\end{eqnarray}
where $\boldsymbol{e}_{\boldsymbol{\chi}^*}=\boldsymbol{e}_{\boldsymbol{y}^*}=\mathbf{0}_{Nr}$ are used. Noticing from \eqref{y*} that there exist $\chi_{i}^*=\chi_{j}^*=\chi^*$ and $y_{i}^*=y_{j}^*=y^*$. Meanwhile, in view of \eqref{barchi3}-\eqref{bary3},
 \begin{eqnarray}
	\chi^*\!&=&\bar{\chi}^*=\!\frac{1}{N}\sum_{i=1}^{N}g_{i}(x_{i}^*)=\chi(\boldsymbol{x}^*),\label{barchi4}\\
	y^*&=&\bar{y}^*=\frac{1}{N}\sum_{i=1}^{N}\nabla_{\chi}f_{i}(x_{i}^*,\chi(\boldsymbol{x}^*)).\label{bary4}
\end{eqnarray}
Based on Assumption~\ref{assumption2}, we compute the gradient of $f(\boldsymbol{x})$ in $\boldsymbol{x}^*$ as follows,
\begin{eqnarray}
	\nabla_{\boldsymbol{x}}f(\boldsymbol{x}^*)&=&\nabla_{\boldsymbol{x}}f(\boldsymbol{x}^*,\mathbf{1}_{N}\otimes\chi^*)+\nabla_{\boldsymbol{x}}g(\boldsymbol{x}^*)[\mathbf{1}_{N}\otimes\frac{1}{N}\sum_{i=1}^{N}\nabla_{\chi_{i}}f_{i}(x_{i}^*,\chi_{i}^*)],\nonumber\\
	&=&\nabla_{\boldsymbol{x}}f(\boldsymbol{x}^*,\mathbf{1}_{N}\otimes\chi^*)+\nabla_{\boldsymbol{x}}g(\boldsymbol{x}^*)\left[\mathbf{1}_{N}\otimes y^*\right],\nonumber\\
	&=& \mathbf{0}_{Nn},\label{graf}
\end{eqnarray}
where \eqref{barchi4} and \eqref{bary4} are used in the second equality and \eqref{x*} is used in the third equality. \eqref{graf} implies that $\boldsymbol{x}^*$ is the optimal solution to problem \eqref{primalproblem}.
\end{proof}
\begin{remark}
The optimal condition~\eqref{y*} implies that the fixed points $\boldsymbol{\chi}^*=1_{N}\otimes\chi^*$ and $\boldsymbol{y}^*=1_{N}\otimes y^*$. Meanwhile, \eqref{barchi4}-\eqref{bary4} means that $\boldsymbol{\chi}^*=1_{N}\otimes \chi(\boldsymbol{x}^*)$ and $\boldsymbol{y}^*=1_{N}\otimes \frac{1}{N}\sum_{i=1}^{N}\nabla_{\chi}f_{i}(x_{i}^*,\chi(\boldsymbol{x}^*))$. Hence, in the D-QAGT algorithm, $\chi_{i}(k)$ is leveraged for agent $i$ to track the global information $\chi(\boldsymbol{x})$ and $y_{i}(k)$ is leveraged for agent $i$ to seek the gradient sum $\frac{1}{N}\sum_{i=1}^{N}\nabla_{\chi}f_{i}(x_{i},\chi(\boldsymbol{x}))$.
\end{remark}

Then, the following lemmas give the convergence analysis framework of the D-QAGT algorithm with two intermediate results.
\begin{lemma}\label{Lemma3}
	Define $\Theta(k)=(\|\boldsymbol{x}(k)-\boldsymbol{x}^*\|;\|\boldsymbol{\chi}(k)-\mathbf{1}_{N}\otimes\bar{\chi}(k)\|;\|\boldsymbol{y}(k)-\mathbf{1}_{N}\otimes\bar{y}(k)\|)$ and $E(k)=[0;2\|\boldsymbol{e}_{\boldsymbol{\chi}}(k)\|;2l_{2}\|\boldsymbol{e}_{\boldsymbol{\chi}}(k)\|+2\|\boldsymbol{e}_{\boldsymbol{y}}(k)\|]$. Under Assumptions \ref{assumption1} and \ref{assumption2} and considering the iterations on \eqref{x}-\eqref{y}, then the following inequality holds
	\begin{eqnarray}
		\Theta(k+1)\leq H(\alpha)\Theta(k)+E(k).\label{lemma33}
	\end{eqnarray}
Furthermore, the spectral radius $\rho(H)<1$.
\end{lemma}
\begin{proof}
	See Appendix~\ref{a51}.
\end{proof}

The $3$-dimensions vector $\Theta(k)$ describes the distance between the iteration $\boldsymbol{x}(k)$, $\boldsymbol{\chi}(k)$, $\boldsymbol{y}(k)$ and the fixed point $\boldsymbol{x}^*$, $\boldsymbol{\chi}^*$, $\boldsymbol{y}^*$, respectively. Lemma \ref{Lemma3} provides the upper bound of $\Theta(k)$, which is related to the stepsize matrix $H(\alpha)$ and the quantization error vector $E(k)$. Hence, we present the following lemma for analyzing the convergence of $\|H^{k}(\alpha)\|$. 
\begin{lemma}\label{Lemma4}
Under Assumptions \ref{assumption1} and \ref{assumption2} and considering the iterations on \eqref{x}-\eqref{y}, for any $\epsilon\in(0,\min(\gamma-\rho(H),2\|H(\alpha)\|))$, the following inequality holds
	\begin{equation}
		\|H^{k}(\alpha)\|\leq c_{3}(\rho(H)+\epsilon)^{k},\label{lemma34}
	\end{equation}
where the constant $c_{3}=3\sqrt{3}\text{max}\{\frac{4\|H(\alpha)\|^2}{\epsilon^2},\frac{\epsilon^2}{4\|H(\alpha)\|^2}\}$.
\end{lemma}
\begin{proof}
	See Appendix~\ref{a52}.
\end{proof}

Note that the chosen of the constant $\epsilon$ can ensure $\rho(H)+\epsilon<1$ due to $\gamma\in(\rho(H),1)$. Then Lemma~\ref{Lemma3} provides the linear convergence of $\|H^{k}(\alpha)\|$. Summarizing the above lemmas, we give the convergence of the D-QAGT algorithm in the following theorem. 

\begin{theorem}\label{theorem1}
Under Assumptions \ref{assumption1} and \ref{assumption2}, if the number of the quantization levels
\begin{flalign}
	2L+1\geq 2\max\{\frac{\zeta C_0+3C_1}{l_0\gamma},\frac{\sqrt{4c_1^2+4c_2^2}}{l_0}\}+1,\label{L}
\end{flalign}
with
\begin{flalign}
&\zeta=\max\{\alpha l_{1}l_2(1+l_{3})+\alpha l_{1}l_3(1+l_{3}),l_2+\alpha l_{1}l_2+\alpha l_{1}l_3+2,\alpha l_{3}l_2+\alpha l_{3}^2+2\},\\
&C_{0}=c_{3}\sqrt{Nnc_{0}^2+4Nr(c_{1}^2+c_{2}^2)}+\frac{c_{3}C_{1}}{\gamma-\rho(M)-\epsilon},\\
&C_{1}=2(l_{2}+1)\sqrt{Nr}l_{0},
\end{flalign}
and $\epsilon\in(0,\min(\gamma-\rho(H),2\|H(\alpha)\|))$, then $\boldsymbol{x}=[x_{i}]_{i\in{\mathcal{V}}}$ generated by the D-QAGT algorithm can converge to the optimizer of problem \eqref{primalproblem} at a linear convergence rate.
\end{theorem}
\begin{proof}
We prove Theorem~\ref{theorem1} by proving the following inequalities hold simultaneously.
\begin{flalign}
		&\|\Theta(k)\|\leq C_{0}\gamma^{k},\label{t1}\\	
		&\|E(k)\|\leq C_{1}\gamma^{k},\label{t2}
\end{flalign} 
where $\Theta(k)$ and $E(k)$ are defined the same as Lemma~\ref{Lemma3}.

When $k=0$, based on the definition of $\Theta(k)$ and the initial values of $x_{i}(0)$, $\chi_{i}(0)$ and $y_{i}(0)$, we have that 
\begin{eqnarray}
	\|\Theta(0)\|\leq \sqrt{Nnc_{0}^2+4Nr(c_{1}^2+c_{2}^2)}.
\end{eqnarray}
It follows from $c_{3}>1$ and $\frac{c_{3}C_{1}}{\gamma-\rho(H)-\epsilon}>0$ that $$\|\Theta(0)\|\leq c_{3}\sqrt{Nnc_{0}^2+4Nr(c_{1}^2+c_{2}^2)}+\frac{c_{3}C_{1}}{\gamma-\rho(H)-\epsilon}=C_{0},$$ which ensures that \eqref{t1} holds at $k=0$.

When $k\in{\mathbb{N}^{+}}$, we assume that \eqref{t1} and \eqref{t2} hold for any $k\leq k_{1},~k_{1}\in{\mathbb{N}^{+}}$. We first prove that \eqref{t1} holds for $k=k_{1}+1$. Iterating \eqref{lemma33} yields
\begin{eqnarray}
	\Theta(k_{1}+1)\leq H^{k_{1}+1}(\alpha)\Theta(0)+\sum_{i=0}^{k_{1}}H^{i}(\alpha)E(k_{1}-i).\label{t11}
\end{eqnarray}
Substituting \eqref{lemma34} into \eqref{t11} and using \eqref{t2} holds for any $k\leq k_{1},~k_{1}\in{\mathbb{N}^{+}}$, we have that
\begin{flalign}
	\|\Theta(k_{1}+1)\|&\leq c_{3}\bar{\gamma}^{k_{1}+1}\|\Theta(0)\|+\sum_{i=0}^{k_{1}}c_{3}\bar{\gamma}^{i}\|E(k_{1}-i)\|,\nonumber\\
	&\leq c_{3}\bar{\gamma}^{k_{1}+1}\|\Theta(0)\|+\sum_{i=0}^{k_{1}}c_{3}\bar{\gamma}^{i}C_{1}\gamma^{k_{1}-i},\nonumber\\
	&\leq(c_{3}\|\Theta(0)\|+\frac{aC_{1}}{\gamma-\bar{\gamma}})\gamma^{k_{1}+1},\label{t12}
\end{flalign}
where $\bar{\gamma}=\rho(H(\alpha))+\epsilon$. Based on the definition of $C_{0}$, it follows that
\begin{flalign}
\|\Theta(k_{1}+1)\|\leq C_{0}\gamma^{k_{1}+1}. \label{k1}
\end{flalign}
Then, we prove \eqref{t2} for $k=k_{1}+1$. Define $$W(k+1)=[\boldsymbol{\chi}(k+1)-\hat{\boldsymbol{\chi}}(k); \boldsymbol{y}(k+1)-\hat{\boldsymbol{y}}(k)].$$
For $k=k_{1}+1$, taking norm on $W(k_{1}+1)$ yields that
\begin{flalign}
\|W(k_{1}+1)\|&=\|\boldsymbol{\chi}(k_{1}+1)-\hat{\boldsymbol{\chi}}(k_{1})\|+\|\boldsymbol{y}(k_{1}+1)-\hat{\boldsymbol{y}}(k_{1})\|,\nonumber\\
&\leq \|\boldsymbol{\chi}(k_1+1)-\boldsymbol{\chi}(k_1)\|+\|\boldsymbol{y}(k_1+1)-\boldsymbol{y}(k_1)\|+\|\boldsymbol{e}_{\boldsymbol{\chi}}(k)\|+\|\boldsymbol{e}_{\boldsymbol{y}}(k)\|.\label{t13}
\end{flalign}
Considering the first term of \eqref{t13}, in light of $\|A-I_{N}\|=\|-L\|\leq 2$, it follows from \eqref{chie} that
\begin{flalign}
	&\|{\boldsymbol{\chi}}(k_1+1)-{\boldsymbol{\chi}}(k_1)\|\nonumber\\
	&\leq\|(A-I_{N})({\boldsymbol{\chi}}(k_1)-1_N\otimes\bar{\chi}(k_{1}))+(I_{N}-A)\boldsymbol{e}_{\boldsymbol{\chi}}(k_1)\|+l_3\|\boldsymbol{x}(k_1+1)-\boldsymbol{x}(k_1)\|,\nonumber\\
	&\leq2\|{\boldsymbol{\chi}}(k_1)-1_N\otimes \bar{\chi}(k_{1})\|+2\|\boldsymbol{e}_{\boldsymbol{\chi}}(k_1)\|+l_3\|\boldsymbol{x}(k_1+1)-\boldsymbol{x}(k_1)\|.\label{t34}
\end{flalign}
Substituting \eqref{xkx1} into \eqref{t34} yields
\begin{flalign}
&\|{\boldsymbol{\chi}}(k_1+1)-{\boldsymbol{\chi}}(k_1)\|\nonumber\\	
	&\leq \alpha l_{1}l_3(1+l_{3})\|\boldsymbol{x}(k_{1})-\boldsymbol{x}^*\|+(\alpha l_{1}l_3+2)\|\boldsymbol{\chi}(k_{1})-\mathbf{1}_{N}\otimes\bar{\chi}(k)\|\nonumber\\
	&\quad+\alpha l_{3}^2\|\boldsymbol{y}(k_{1})-\mathbf{1}_{N}\otimes\bar{y}(k_{1})\|+2\|\boldsymbol{e}_{\boldsymbol{\chi}}(k_1)\|.\label{t15}
\end{flalign}
Similarly, considering the second term of \eqref{t13} and using \eqref{xkx1} again, it follows from \eqref{ye} that
\begin{flalign}
	&\|\boldsymbol{y}(k_1+1)-{\boldsymbol{y}}(k_1)\|\nonumber\\
&\leq\|(I_{N}-A)({\boldsymbol{y}}(k_1)-1_N\otimes \bar{y}(k_1))\|+\|(A-I_{N})\boldsymbol{e}_{\boldsymbol{y}}(k_1)+\nabla_{\boldsymbol{\chi}} f(\boldsymbol{x}(k_1+1),{\boldsymbol{\chi}}(k_1+1))\nonumber\\
&\quad-\nabla_{\boldsymbol{\chi}} f(\boldsymbol{x}(k_1),{\boldsymbol{\chi}}(k_1))\|,\nonumber\\
&\leq2\|{\boldsymbol{y}}(k_1)-1_N\otimes \bar{y}(k_1)\|+2\|\boldsymbol{e}_{\boldsymbol{y}}(k_1)\|+l_2\|\boldsymbol{x}(k_1+1)-\boldsymbol{x}(k_1)\|+l_2\|{\boldsymbol{\chi}}(k_1+1)-{\boldsymbol{\chi}}(k_1)\|,\nonumber\\
&\leq\alpha l_{1}l_2(1+l_{3})\left\|\boldsymbol{x}(k_1)-\boldsymbol{x}^{*}\right\|+(l_2+\alpha l_{1}l_2)\|{\boldsymbol{\chi}}(k_1)-1_N\otimes \bar{\chi}(k_{1})\|\nonumber\\
&\quad+(2+\alpha l_{3}l_2)\big\|\boldsymbol{y}(k_1)-1_N\otimes \bar{y}(k_1)\big\|+2\|\boldsymbol{e}_{\boldsymbol{y}}(k_1)\|.\label{t16}
\end{flalign}
Combing with \eqref{t15} and \eqref{t16}, we bound $\|W(k_{1}+1)\|$ as 
\begin{flalign}
&\|W(k_{1}+1)\|\nonumber\\
&\leq (\alpha l_{1}l_2(1+l_{3})+\alpha l_{1}l_3(1+l_{3}))\left\|\boldsymbol{x}(k_1)-\boldsymbol{x}^{*}\right\|+(l_2+\alpha l_{1}l_2+\alpha l_{1}l_3+2)\nonumber\\
&\quad\|{\boldsymbol{\chi}}(k_1)-1_N\otimes \bar{\chi}(k_{1})\|+(\alpha l_{3}l_2+\alpha l_{3}^2+2)\left\|\boldsymbol{y}(k_1)-1_N\otimes \bar{y}(k_1)\right\|\nonumber\\
&\quad+3\|\boldsymbol{e}_{\boldsymbol{y}}(k_1)\|+3\|\boldsymbol{e}_{\boldsymbol{\chi}}(k_1)\|.\label{t17}
\end{flalign}
Since \eqref{t1} and \eqref{t2} hold for any $k=k_{1}$, we conclude that
\begin{flalign}
	\|W(k+1)\|&\leq \zeta\Theta(k_1)+3(\|\boldsymbol{e}_{\boldsymbol{\chi}}(k_1)\|+\|\boldsymbol{e}_{\boldsymbol{y}}(k_1)\|),\nonumber\\
	&\leq(\zeta C_0+3C_1)\gamma^{k_1}.
\end{flalign}
It further yields
\begin{equation}\label{addd}
	\frac{\|W(k_{1}+1)\|_{\infty}}{l(k_1+1)}\leq \frac{\|W(k_{1}+1)\|}{l(k_1+1)}\leq\frac{\zeta C_0+3C_1}{l_0\gamma},
\end{equation}
which means that $\frac{\|W(k_{1}+1)\|_{\infty}}{l(k_1+1)}\leq L$ and the quantizer is unsaturated at $k=k_{1}+1$ by recalling from the definition of $L$ in \eqref{L}. Therefore, $\|\boldsymbol{e}_{\boldsymbol{\chi}}(k_1+1)\|_{\infty}\leq l(k_{1}+1)$ and $\|\boldsymbol{e}_{\boldsymbol{y}}(k_1+1)\|_{\infty}\leq l(k_{1}+1)$. Together with the definition of $E(k)$, one has that
\begin{flalign}
\|E(k_{1}+1)\|&\leq (2+2l_{2})\|\boldsymbol{e}_{\boldsymbol{\chi}}(k_1+1)\|+2\|\boldsymbol{e}_{\boldsymbol{y}}(k_1+1)\|,\nonumber\\
&\leq(2+2l_{2})(\|\boldsymbol{e}_{\boldsymbol{\chi}}(k_1+1)\|+\|\boldsymbol{e}_{\boldsymbol{y}}(k_1+1)\|),\nonumber\\
&\leq(1+l_{2})\sqrt{Nr}l(k_{1}+1)=C_{1}\gamma^{k_{1}+1}.\label{k2}
\end{flalign}
To sum up, based on \eqref{k1} and \eqref{k2}, we conclude that \eqref{t1} and \eqref{t2} hold for $k=k_{1}+1$. Following by the principle of induction, we can conclude that \eqref{t1} and \eqref{t2} hold for any $k\in{\mathbb{N}}$. Based on the definition of $\Theta(k)$, it can guarantee the linear convergence rate of $\boldsymbol{x}(k)$, which ends the proof. 
\end{proof}
\begin{remark}
It should be noted that $\|\boldsymbol{x}(k)-\boldsymbol{x}^*\|\leq\|\Theta\|$. Thus, $\|\Theta(k)\|\leq C_{0}\gamma^{k}$ in~\eqref{t1} implies that $\lim_{k\rightarrow \infty} \frac{\|\boldsymbol{x}(k+1)-\boldsymbol{x}^*\|}{\|\boldsymbol{x}(k)-\boldsymbol{x}^*\|}\leq \gamma<1$, which ensures the linear convergence of $\boldsymbol{x}(k)$ at rate $\gamma$. Based on the definition of $\gamma\in\left(\rho(H),1\right)$, the linear convergence rate of the D-QAGT algorithm is no less than $\rho(H)$. In addition, according to the selection of the quantization levels $2L+1$ in~\eqref{L}, the required maximum bandwidth for ensuring the linear convergence of the D-QAGT is $\mathcal{B}=\big\lceil\log_{2}(2\max\{\frac{\zeta C_0+3C_1}{l_0\rho(H)},\frac{\sqrt{4c_1^2+4c_2^2}}{l_0}\}+1)\big\rceil$.
\end{remark}

\section{Simulation}

To evaluate the performance of the proposed algorithm, we verify its communication saving in the optimal placement problem. In an optimal placement problem, assume that there are $M=5$ entities, which are located at $r_{1}=(3;5),~r_{2}=(6;9),~r_{3}=(9;8),~r_{4}=(6;2)$ and $r_{5}=(9;2)$. There are distributed across $N=5$ free entities, each of that privately knows some of the fixed $M$ entities. The aim is to determine the optimal position $\boldsymbol{x}=[x_{i}]_{i\in\{1,\cdots,5\}}$ of the free entities for minimizing the sum of all distance from the current position of each free entity
to a corresponding fixed entity’s position and the distances
from each agent to the weighted center of all
free entities.  In this case, the cost function of each free entity is modeled as follows
\begin{equation}
	f_{i}(x_{i},\chi(\boldsymbol{x}))=\gamma_{i}\|x_{i}-r_{i}\|^2+\|x_{i}-\chi(\boldsymbol{x})\|^2,~i\in\{1,\cdots,5\},
\end{equation}
where $\gamma_{i}=100$ represents the weight constant and $\chi(\boldsymbol{x})=\sqrt{\sum_{i=1}^5 x_{i}^2/N}$. The communication graph among the free entities is randomly chosen to be strongly connected and the communication channel between free entities exists finite bandwidth constraints. The stepsize is chosen as $\alpha=0.01$ in the D-QAGT algorithm. 

The scaling functions is chosen as $s_{\boldsymbol{\chi}}(k)=s_{\boldsymbol{y}}(k)=10e^{-0.1k}$ and the quantization levels is chosen as $L=10$. Fig. 1 shows that the evolution of $x_{i}(k)$ associated with iterations $k=100$. The results shows that all free entities in the plane converge to their optimal positions $\boldsymbol{x}^*=[x_{i}^*]_{i\in\{1,\cdots,5\}}$. 
\begin{figure}
	\label{FF1}
	\centering
	\subfigure[The first coordinate of $x_{i}(k)$.]{\includegraphics[height=3cm,width=6cm]{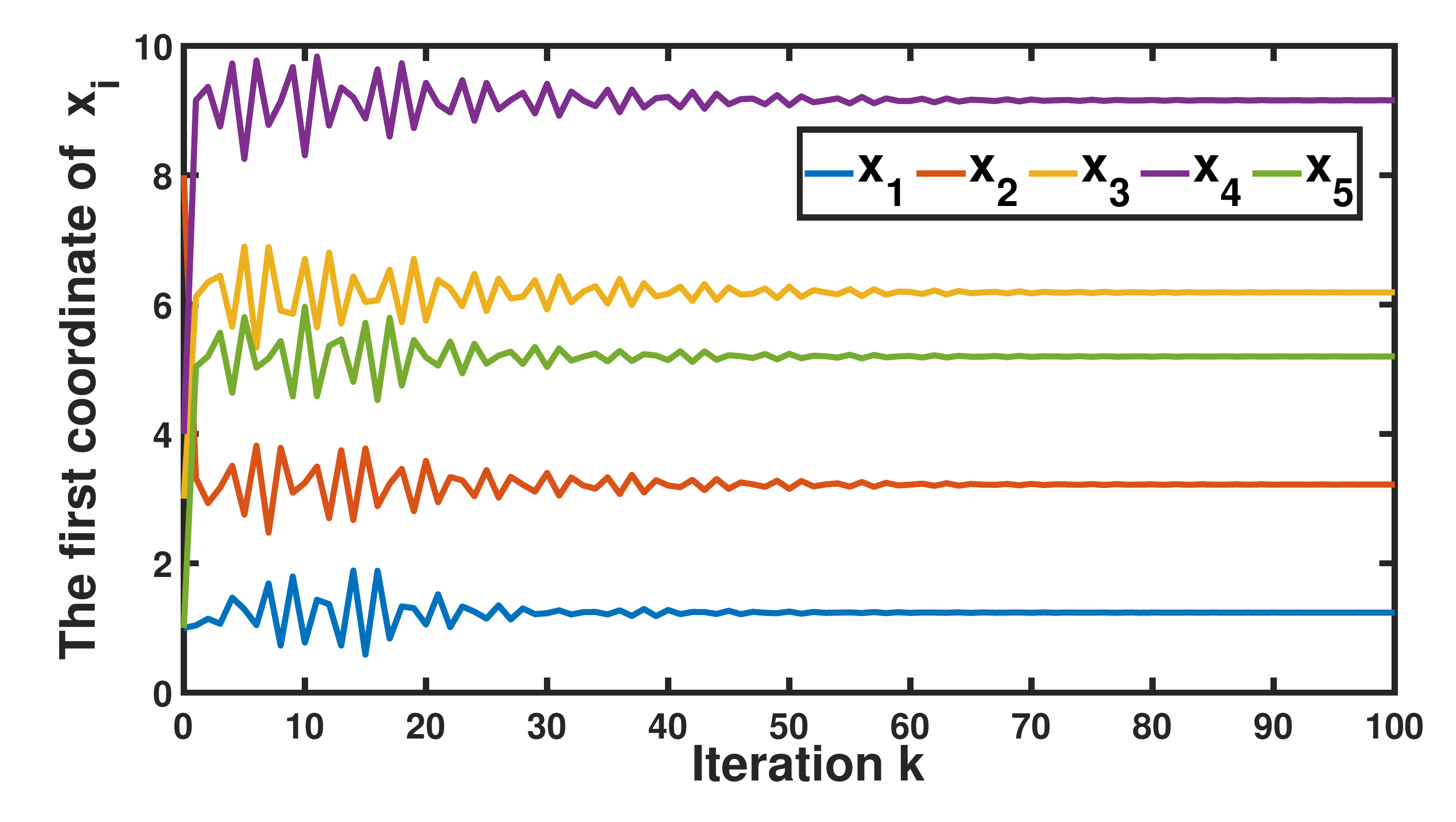}}\label{fig4.1}
	\quad
	\subfigure[The second  coordinate of $x_{i}(k)$.]{\includegraphics[height=3cm,width=6cm]{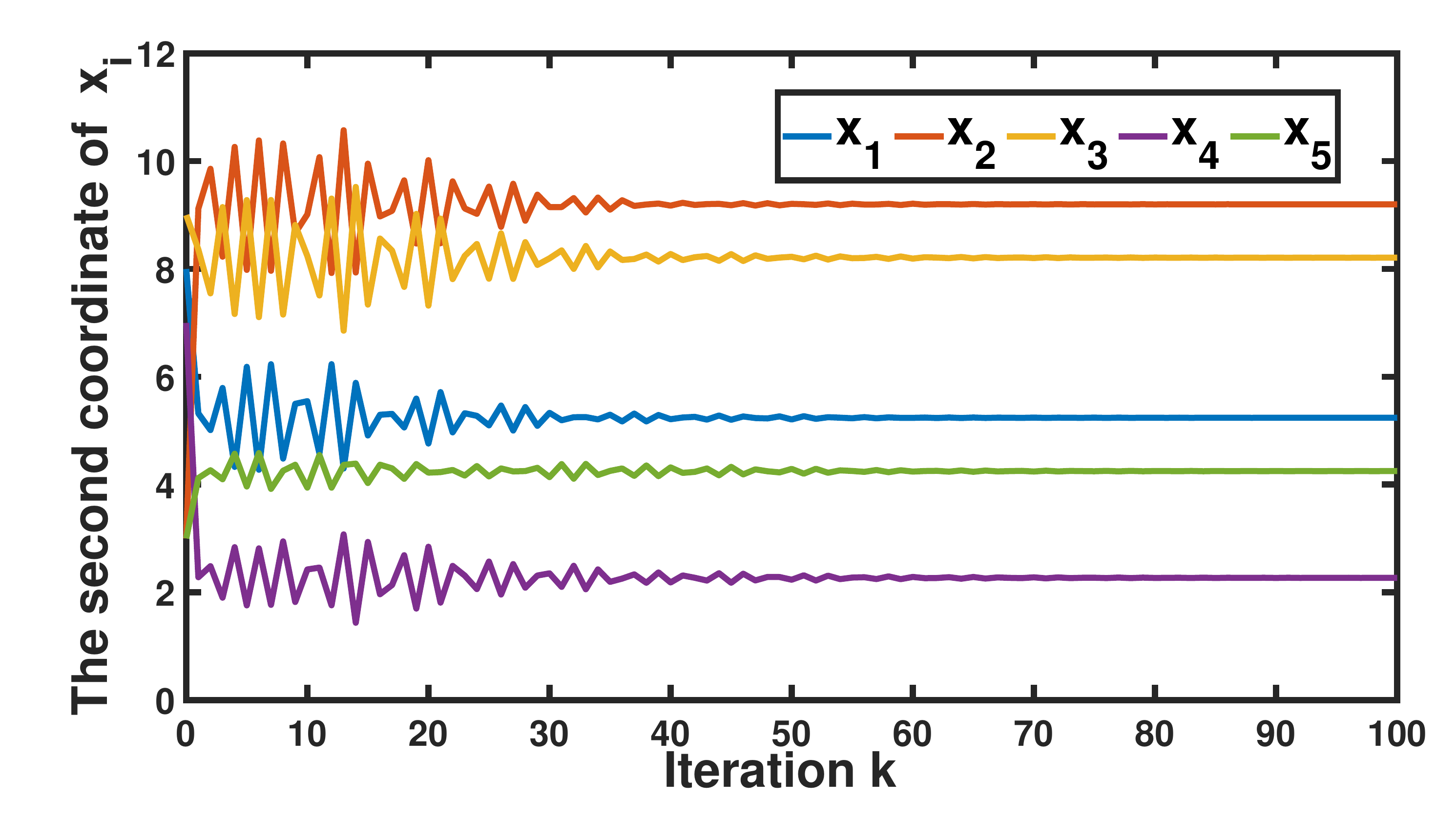}}\label{fig4.2}
	\caption{The evolution of $x_{i}(k)\in{\mathbb{R}^{2}},~i\in\{1,\cdots,5\}$. The optimal positions of each agent are $x_{1}^*=(1.2376;5.2393)$, $x_{2}^*=(3.2178,9.1997)$, $x_{3}^*=(6.1881;8.2096)$, $x_{4}^*=(9.1584,2.2690)$ and $x_{5}^*=(5.1980,4.2492)$.}
\end{figure}
Fig. 2 shows that the evolution of $\chi_{i}(k)$. The results shows that the estimate $\chi_{i}(k)$ of each free entity converges to the optimal $\chi(\boldsymbol{x}^*)=\sqrt{\sum_{i=1}^5 (x_{i}^*)^2/N}$.
\begin{figure}
	\centering
	\subfigure[The first coordinate of $\chi_{i}(k)$ v.s. iteration.]{\includegraphics[height=3cm,width=6cm]{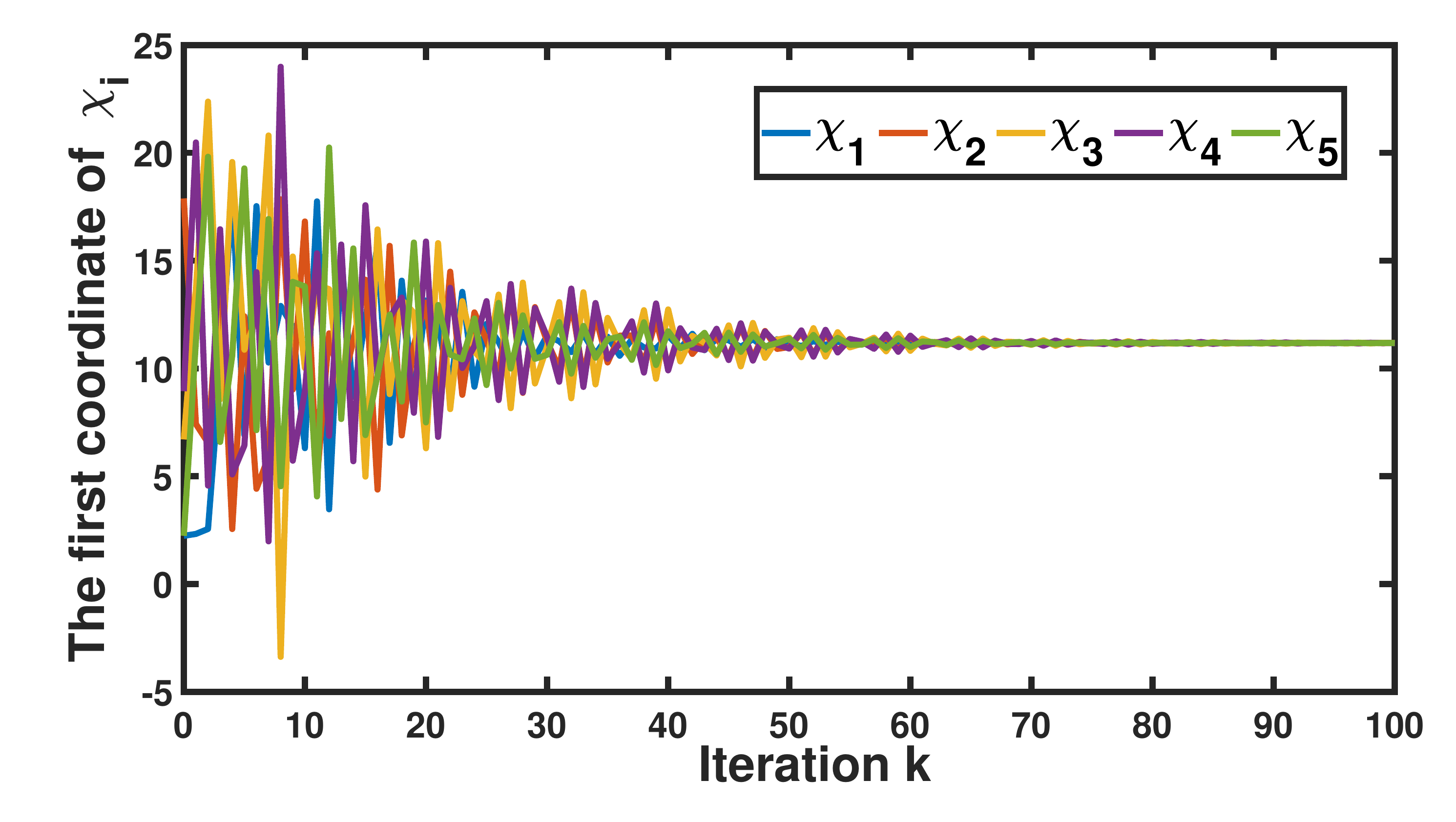}}\label{fig4.3}
	\quad
	\subfigure[The second coordinate of $\chi_{i}(k)$ v.s. iteration.]{\includegraphics[height=3cm,width=6cm]{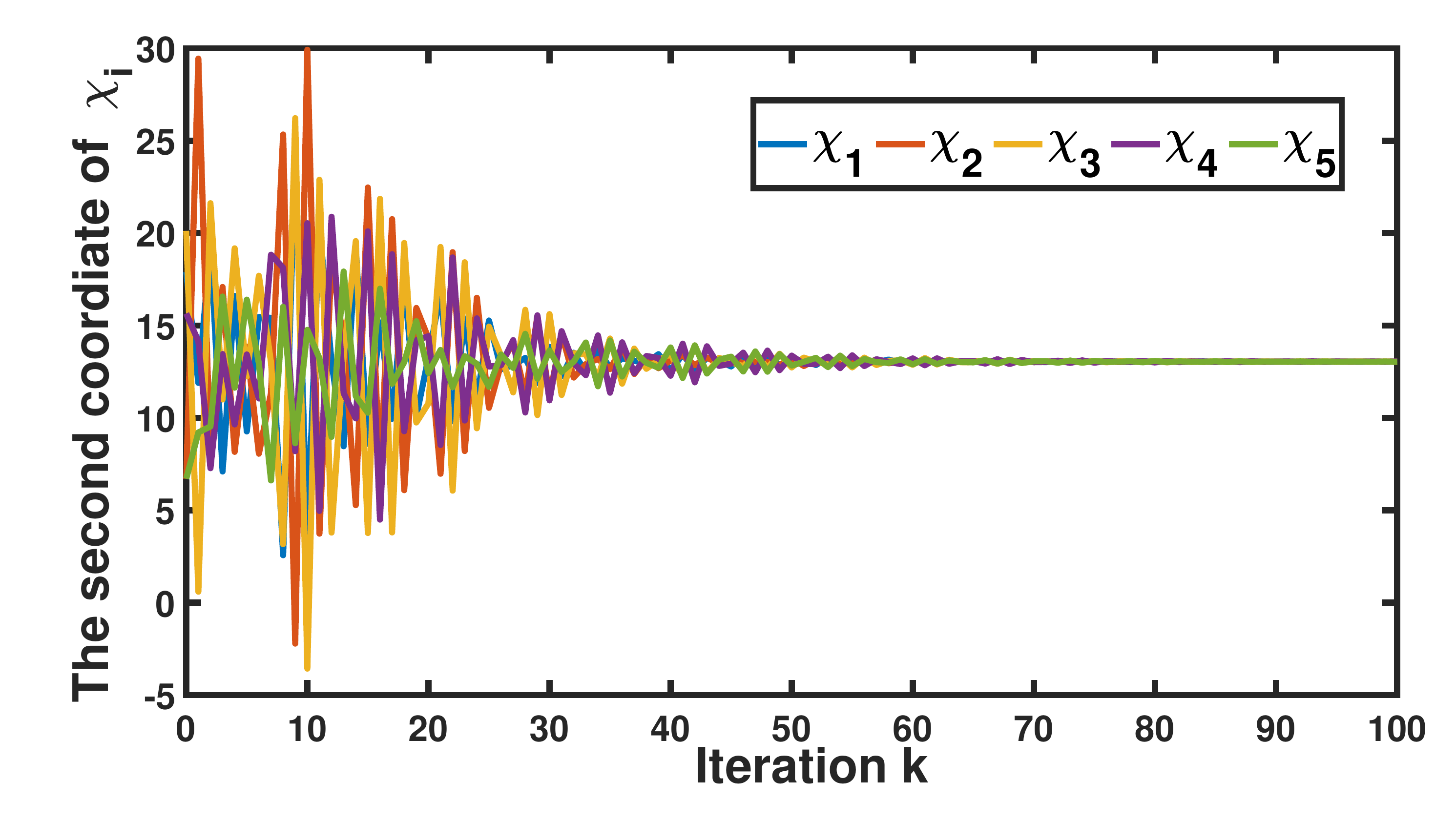}}\label{fig4.4}
	\caption{The evolution of the estimate $\chi_{i}(k)\in{\mathbb{R}^{2}},~i\in\{1,\cdots,5\}$. The optimal aggregative information is $\chi(\boldsymbol{x}^*)=(11.1803,13.0437)$.}
\end{figure}

In order to show the linear convergence, we take a performance index $J$ as $J(t)=e^{\gamma k}\|f(\boldsymbol{x})-f(\boldsymbol{x}^*)\|$ with $\gamma=0.1$. The trajectory of $J(t)$ for the D-QAGT is shown in Fig. 3(a). The results illustrates linear convergence of D-QAGT with a rate no less than the constant $0.1$. To compare the convergence with the related D-AGT algorithm in~\cite{lixiuxian}, in which the communication channel is ideal, the evolutions of $f(\boldsymbol{x})=\sum_{i=1}^{5}f_{i}(\boldsymbol{x})$ for D-AGT and D-QAGT are shown in Fig. 3(b). The results shows that the influence of quantization on the convergence speed is slight under the quantization levels $L=10$.
\begin{figure}
	\centering
	\subfigure[The perfmance function $J(t)$ v.s. interation]{\includegraphics[height=3cm,width=6cm]{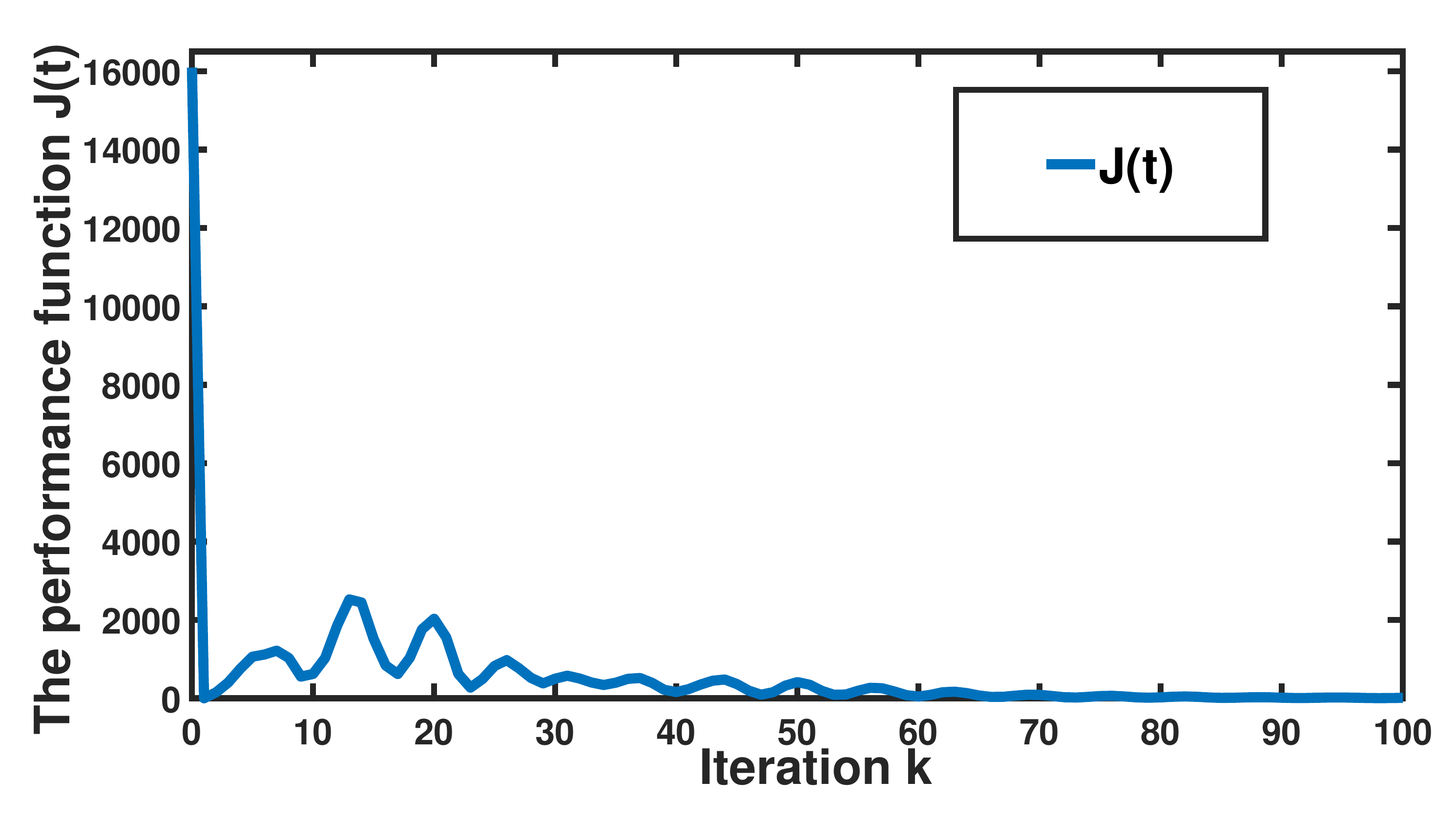}}\label{fig4.5}
	\quad
	\subfigure[The global functions $J(t)$ v.s. iteration.]{\includegraphics[height=3cm,width=6cm]{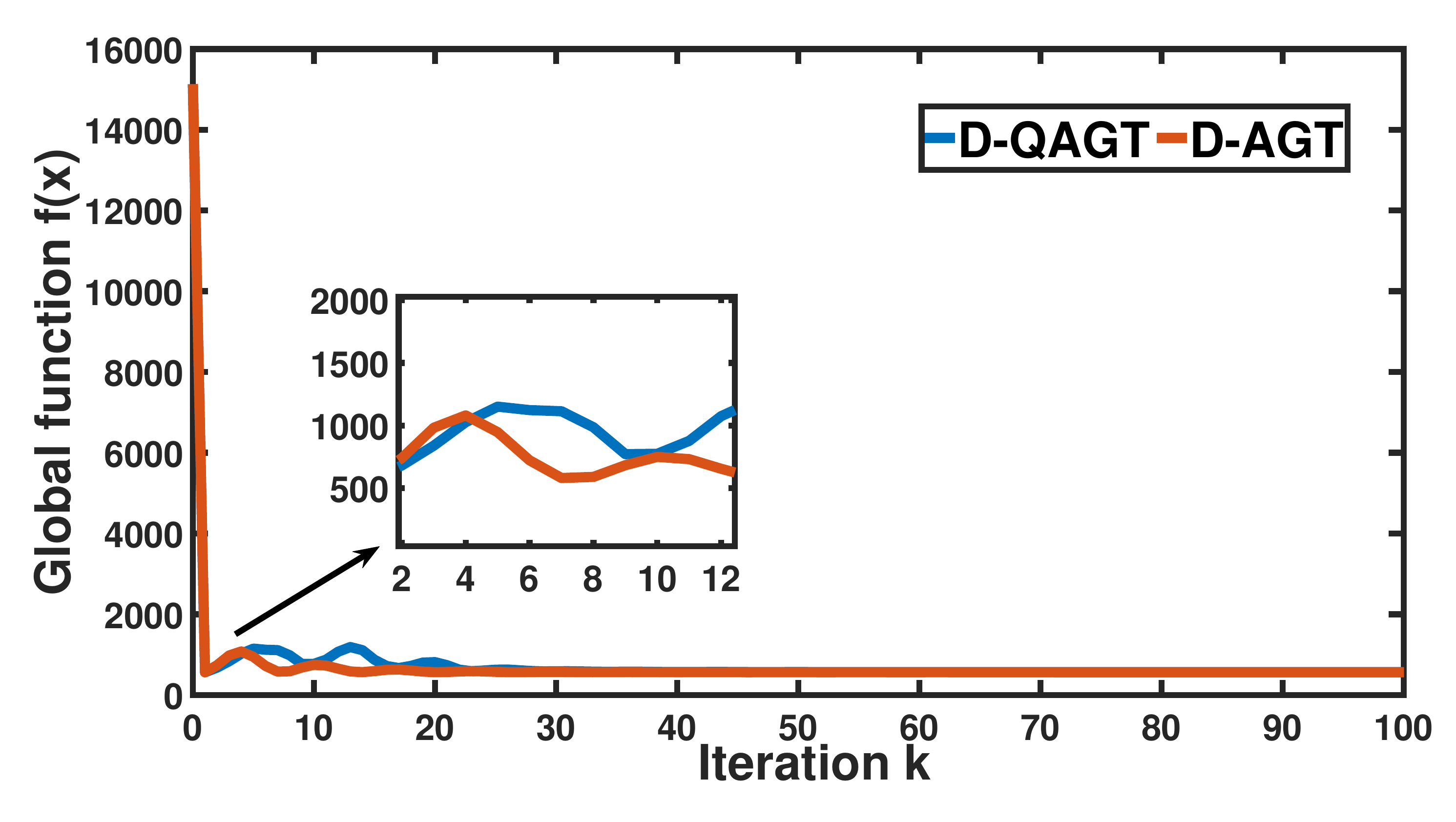}}\label{fig4.6}
	\caption{(a) represents the trajectory of the performance function $J(t)$. (b) represents the convergence trajectories of $f(\boldsymbol{x})$ for Q-DAGT (blue line) and D-AGT (red line), respectively.}
\end{figure}

\section{Conclusion}
This paper proposed a novel distributed quantized algorithm to solve the aggregative optimization problem, in which agents collectively minimize the sum of the local cost function that depends on the global aggregative variable. Particularly, we combined gradient descent and variables tracking methods to estimate the global aggregative variable, and also introduced the quantization technology to overcome communication bottleneck. By using dynamic encoder-decoder schemes, the influence of quantization errors is eliminated so that the linear convergence of the proposed Q-DGAT algorithms is achieved.  Future works would include exploiting multiple information compression technologies for distributed aggregative optimization problems, such as sparsification and even-trigger technologies.
\section{Appendix}
\subsection{Proof of Lemma~\ref{Lemma3}}\label{a51}
For explicit illustration, we prove the Lemma~\ref{Lemma3} by the following two steps. 

\textbf{Step 1.} We first bound $\|\boldsymbol{x}(k+1)-\boldsymbol{x}^*\|$, $\|\boldsymbol{\chi}(k+1)-\mathbf{1}_{N}\otimes\bar{\chi}(k+1)\|$ and $\|\boldsymbol{y}(k+1)-\mathbf{1}_{N}\otimes\bar{y}(k+1)\|$, respectively. 
\begin{itemize}
\item Invoking \eqref{xx} into $\|\boldsymbol{x}(k+1)-\boldsymbol{x}^*\|$ yields that
\begin{flalign}
	&\left\|\boldsymbol{x}(k+1)-\boldsymbol{x}^{*}\right\|\nonumber \\
	&=\big\|\boldsymbol{x}(k)-\boldsymbol{x}^{*}-\alpha\left[\nabla_{\boldsymbol{x}} f\left(\boldsymbol{x}(k), \boldsymbol{\chi}(k)\right)+\nabla_{\boldsymbol{x}}g\left(\boldsymbol{x}(k)\right)\boldsymbol{y}(k)\right]\big\|, \nonumber\\
	&\leq \big\|\boldsymbol{x}(k)\!\!-\!\boldsymbol{x}^{*}\!\!-\!\alpha \big[\nabla_{\boldsymbol{x}} f\left(\boldsymbol{x}(k), \mathbf{1}_{N} \!\otimes\! \bar{\chi}(k)\right)\!\!+\!\nabla_{\boldsymbol{x}}g\left(\boldsymbol{x}(k)\right)\![\mathbf{1}_{N} \!\otimes\! \frac{1}{N} \sum_{i=1}^{N} \nabla_{\bar{\chi}} f_{i}\left(x_{i}(k), \bar{\chi}(k)\right)]\big]\nonumber \\
	&\quad+\alpha \nabla_{\boldsymbol{x}} f\left(\boldsymbol{x}^{*}\right) \big\|+\alpha \big\| \nabla_{\boldsymbol{x}} f\left(\boldsymbol{x}(k), \boldsymbol{\chi}(k)\right)+\nabla_{\boldsymbol{x}}g\left(\boldsymbol{x}(k)\right)( \mathbf{1}_{N} \otimes \bar{y}(k))\nonumber\\
	&\quad-\nabla_{\boldsymbol{x}} f\left(\boldsymbol{x}(k), \mathbf{1}_{N} \otimes \bar{\chi}(k)\right) -\nabla_{\boldsymbol{x}}g\left(\boldsymbol{x}(k)\right)[ \mathbf{1}_{N} \otimes \frac{1}{N} \sum_{i=1}^{N} \nabla_{\bar{\chi}} f_{i}\left(x_{i}(k), \bar{\chi}(k)\right)]\big\|\nonumber\\
	&\quad+\alpha\big\|\nabla_{\boldsymbol{x}}g\left(\boldsymbol{x}(k)\right) \boldsymbol{y}(k)-\nabla_{\boldsymbol{x}}g\left(\boldsymbol{x}(k)\right)(\mathbf{1}_{N} \otimes \bar{y}(k))\big\|. \label{x_x*}
\end{flalign}
Following from Lemma 3 in \cite{lixiuxian}, for a $\mu$-strongly convex and $l$-smooth function $f: \mathbb{R}^{n}\rightarrow \mathbb{R}$, there is $$\|x-\alpha \nabla_{x}f(x)-(y-\alpha \nabla_{y{\tiny }} f(y))\| \leq(1-\mu \alpha)\|x-y\|,~\forall x,y\in{\mathbb{R}^{n}},$$ with $\alpha\in(0,1/l]$. Recalling \eqref{barchi3} that $\bar{\chi}(k)=\chi(\boldsymbol{x}(k))$ and using 
$\nabla_{\boldsymbol{x}}f(\boldsymbol{x}(k))=\nabla_{\boldsymbol{x}}f(\boldsymbol{x}(k),\mathbf{1}_{N}\otimes\bar{\chi}(k))+\nabla_{\boldsymbol{x}}g(\boldsymbol{x}(k))\big[\mathbf{1}_{N}\otimes\frac{1}{N}\sum_{i=1}^{N}\nabla_{\bar{\chi}}f_{i}(x_{i}(k),\bar{\chi}(k))\big]$, we bound the first term of the right of \eqref{x_x*} as follows,
\begin{flalign}
&\big\|\boldsymbol{x}(k)\!\!-\!\alpha \big[\nabla_{\boldsymbol{x}} f\left(\boldsymbol{x}(k), \mathbf{1}_{N} \otimes \bar{\chi}(k)\right)\!+\!\nabla_{\boldsymbol{x}}g\left(\boldsymbol{x}(k)\right)[\mathbf{1}_{N} \otimes \frac{1}{N} \sum_{i=1}^{N} \nabla_{\bar{\chi}} f_{i}\left(x_{i}(k), \bar{\chi}(k)\right)]\big]\nonumber \\
&\quad-(\boldsymbol{x}^{*}-\alpha \nabla_{\boldsymbol{x}} f\left(\boldsymbol{x}^{*}\right)) \big\|\leq(1-\mu\alpha)\|\boldsymbol{x}(k)-\boldsymbol{x}^{*}\|.\label{first}
\end{flalign}
Similarly, using \eqref{bary3} and Assumption~\ref{assumption2}, the second term of the right of \eqref{x_x*} is bounded as
\begin{flalign}
		&\alpha \big\| (\nabla_{\boldsymbol{x}} f\left(\boldsymbol{x}(k), \boldsymbol{\chi}(k)\right)+\nabla_{\boldsymbol{x}}g\left(\boldsymbol{x}(k)\right)( \mathbf{1}_{N} \otimes \bar{y}(k)))-(\nabla_{\boldsymbol{x}} f\left(\boldsymbol{x}(k), \mathbf{1}_{N} \otimes \bar{\chi}(k)\right) \nonumber\\
		&\quad+\nabla_{\boldsymbol{x}}g\left(\boldsymbol{x}(k)\right)[ \mathbf{1}_{N} \otimes \frac{1}{N} \sum_{i=1}^{N} \nabla_{\bar{\chi}} f_{i}\left(x_{i}(k), \bar{\chi}(k)\right)])\big\|\nonumber\\
		&\leq \alpha l_{1}\|\boldsymbol{\chi}(k)-\mathbf{1}_{N}\otimes \bar{\chi}(k)\|.\label{second}
\end{flalign}
Combining with \eqref{first} and \eqref{second}, in light of $\|\nabla_{\boldsymbol{x}}g\left(\boldsymbol{x}(k)\right)\|\leq l_{3}$, we conclude that
\begin{flalign}
	&\left\|\boldsymbol{x}(k+1)-\boldsymbol{x}^{*}\right\|\leq(1-\mu \alpha)\left\|\boldsymbol{x}(k)-\boldsymbol{x}^{*}\right\|+\alpha l_{1}\|\boldsymbol{\chi}(k)-\mathbf{1}_{N}\otimes\bar{\chi}(k)\|\nonumber\\
	 &\quad+\alpha l_{3}\|\boldsymbol{y}(k)-\mathbf{1}_{N}\otimes\bar{y}(k)\|.
\end{flalign}

\item Invoking \eqref{chie} into $\boldsymbol{\chi}(k+1)-\mathbf{1}_{N}\otimes\bar{\chi}(k+1)$ yields that
\begin{flalign}
&\big\|\boldsymbol{\chi}(k+1)-(\frac{1}{N}\mathbf{1}_{N}\mathbf{1}_{N}^{T}\otimes I_{r})\boldsymbol{\chi}(k+1)\big\|\nonumber\\
&=\big\|\boldsymbol{L}\boldsymbol{e}_{\boldsymbol{\chi}}(k)+g(\boldsymbol{x}(k+1))-g(\boldsymbol{x}(k))+(A\otimes I_{r})\boldsymbol{\chi}(k)-(\frac{1}{N}\mathbf{1}_{N}\mathbf{1}_{N}^{T}\otimes I_{r})\nonumber\\
&\quad\big[\boldsymbol{L}\boldsymbol{e}_{\boldsymbol{\chi}}(k)+g(\boldsymbol{x}(k+1))-g(\boldsymbol{x}(k))+(A\otimes I_{r})\boldsymbol{\chi}(k)\big]\big\|,\nonumber\\
&\leq \big\|(A\otimes I_{r})\boldsymbol{\chi}(k)-(\frac{1}{N}\mathbf{1}_{N}\mathbf{1}_{N}^{T}\otimes I_{r})(A\otimes I_{r})\boldsymbol{\chi}(k)\big\|+\big\|I_{Nr}-(\frac{1}{N}\mathbf{1}_{N}\mathbf{1}_{N}^{T}\otimes I_{r})\big\|\nonumber\\
&\quad\big\|\boldsymbol{L}\boldsymbol{e}_{\boldsymbol{\chi}}(k)\big\|+\big\|I_{Nr}-(\frac{1}{N}\mathbf{1}_{N}\mathbf{1}_{N}^{T}\otimes I_{r}))\big\|\big\|g(\boldsymbol{x}(k+1))-g(\boldsymbol{x}(k))\big\|,	
\end{flalign}
where the fact that $I_{Nr}-\boldsymbol{L}=A\otimes I_{r}$ and $\mathbf{1}_{N}\otimes \bar{\chi}(k)=(\frac{1}{N}\mathbf{1}_{N}\mathbf{1}_{N}^{T}\otimes I_{r})\boldsymbol{\chi}(k)$ are used. Based on Lemma~\ref{Lemma1} and $\|I_{Nr}-\frac{1}{N}\mathbf{1}_{N}\mathbf{1}_{N}^{T}\otimes I_{r}\|=1$, it follows that
\begin{flalign}
		&\big\|\boldsymbol{\chi}(k+1)-(\frac{1}{N}\mathbf{1}_{N}\mathbf{1}_{N}^{T}\otimes I_{r})\boldsymbol{\chi}(k+1)\big\|\nonumber\\
		&\leq \kappa\big\|\boldsymbol{\chi}(k)-(\frac{1}{N}\mathbf{1}_{N}\mathbf{1}_{N}^{T}\otimes I_{r})\boldsymbol{\chi}(k)\big\|+\big\|I_{N}-A\big\|\big\|\boldsymbol{e}_{\boldsymbol{\chi}}(k)\big\|+\big\|g(\boldsymbol{x}(k+1))-g(\boldsymbol{x}(k))\big\|,\nonumber\\
		&\leq \kappa\big\|\boldsymbol{\chi}(k)-(\frac{1}{N}\mathbf{1}_{N}\mathbf{1}_{N}^{T}\otimes I_{r})\boldsymbol{\chi}(k)\big\|+2\|\boldsymbol{e}_{\boldsymbol{\chi}}(k)\|+l_{3}\|\boldsymbol{x}(k+1)-\boldsymbol{x}(k)\|,\label{chikchi}
\end{flalign}
where $\kappa=\|A-1/N\mathbf{1}_{N}\mathbf{1}_{N}^{T}\|<1$. We now compute the upper bound of $\|\boldsymbol{x}(k+1)-\boldsymbol{x}(k)\|$. Invoking \eqref{xx} into $\|\boldsymbol{x}(k+1)-\boldsymbol{x}(k)\|$ yields that
\begin{flalign}
	&\left\|\boldsymbol{x}(k+1)-\boldsymbol{x}(k)\right\|=\alpha\left\|\nabla_{\boldsymbol{x}} f\left(\boldsymbol{x}(k), \boldsymbol{\chi}(k)\right)+\nabla_{\boldsymbol{x}} g\left(\boldsymbol{x}(k)\right)\boldsymbol{y}(k)\right\| \nonumber\\
	&\leq \alpha \big\| \nabla_{\boldsymbol{x}} f\left(\boldsymbol{x}(k), \boldsymbol{\chi}(k)\right)\!\!+\!\nabla_{\boldsymbol{x}} g\left(\boldsymbol{x}(k)\right)(\frac{1}{N}\mathbf{1}_{N}\mathbf{1}_{N}^{T}\otimes I_{r})\boldsymbol{y}(k)\!\!-\!\nabla_{\boldsymbol{x}} f\left(\boldsymbol{x}^{*}, \mathbf{1}_{N} \otimes \chi^{*}\right)\nonumber\\
	&\quad \!\!-\!\nabla_{\boldsymbol{x}}g\left(\boldsymbol{x}^*\right)\big[\mathbf{1}_{N} \!\!\otimes\! \frac{1}{N} \sum_{i=1}^{N} \nabla_{\chi_{i}} f_{i}\left(x_{i}^{*}, \chi_{i}^{*}\right)\big] \big\|\!\!+\!\alpha\big\|\nabla_{\boldsymbol{x}} g\left(\boldsymbol{x}(k)\right)\big(\boldsymbol{y}(k)\!\!-\!(\frac{1}{N}\mathbf{1}_{N}\mathbf{1}_{N}^{T}\!\!\otimes\! I_{r})\boldsymbol{y}(k)\big)\big\|,\nonumber\\
	&=\alpha \big\| \nabla_{\boldsymbol{x}} f\left(\boldsymbol{x}(k), \boldsymbol{\chi}(k)\right)+\nabla_{\boldsymbol{x}} g\left(\boldsymbol{x}(k)\right)(\mathbf{1}_{N}\otimes \bar{y}(k))-\nabla_{\boldsymbol{x}} f\left(\boldsymbol{x}^{*}, \mathbf{1}_{N} \otimes \chi^{*}\right) -\nabla_{\boldsymbol{x}}g\left(\boldsymbol{x}^*\right)\nonumber\\
	&\quad\big[\mathbf{1}_{N} \otimes \frac{1}{N} \sum_{i=1}^{N} \nabla_{\chi_{i}} f_{i}\left(x_{i}^{*}, \chi_{i}^{*}\right)\big] \big\|+\alpha\big\|\nabla_{\boldsymbol{x}} g\left(\boldsymbol{x}(k)\right)(\boldsymbol{y}(k)-\mathbf{1}_{N}\otimes \bar{y}(k))\big\|.\nonumber
\end{flalign}
It follows from Assumption~\ref{assumption2} that
\begin{flalign}
	&\left\|\boldsymbol{x}(k+1)-\boldsymbol{x}(k)\right\|\nonumber\\
	&\leq \alpha l_{1}\left(\left\|\boldsymbol{x}(k)-\boldsymbol{x}^{*}\right\|+\left\|\boldsymbol{\chi}(k)-\mathbf{1}_{N} \otimes \chi^{*}\right\|\right)+\alpha l_{3}\left\|\boldsymbol{y}(k)-\mathbf{1}_{N}\otimes \bar{y}(k)\right\|,\nonumber\\
	&\leq \alpha l_{1}\big(\left\|\boldsymbol{x}(k)\!-\!\boldsymbol{x}^{*}\right\|\!+\!\big\|\boldsymbol{\chi}(k)\!-\!(\frac{1}{N}\mathbf{1}_{N}\mathbf{1}_{N}^{T}\otimes I_{r})\boldsymbol{\chi}(k)\big\|\big)+\alpha l_{3}\left\|\boldsymbol{y}(k)-\mathbf{1}_{N}\otimes \bar{y}(k)\right\|\nonumber\\
	&\quad+\alpha l_{1}\big\|(\frac{1}{N}\mathbf{1}_{N}\mathbf{1}_{N}^{T}\otimes I_{r})\boldsymbol{\chi}(k)-\mathbf{1}_{N} \otimes \chi^{*}\big\|.\label{xkx}
\end{flalign}
Using again $(\frac{1}{N}\mathbf{1}_{N}\mathbf{1}_{N}^{T}\otimes I_{r})\boldsymbol{\chi}(k)=\mathbf{1}_{N}\otimes \bar{\chi}(k)$, we compute the upper bound of the last term in \eqref{xkx} as follows,
\begin{flalign}
	&\big\|(\frac{1}{N}\mathbf{1}_{N}\mathbf{1}_{N}^{T}\otimes I_{r})\boldsymbol{\chi}(k)-\mathbf{1}_{N} \otimes \chi^{*}\big\|^2=\big\|\mathbf{1}_{N}\otimes(\bar{\chi}(k)-\chi^*)\big\|^2 \nonumber\\
	&=N\big\|\frac{1}{N} \sum_{i=1}^{N}\left(g_{i}\left(x_{i}(k)\right)-g_{i}\left(x_{i}^*\right)\right)\big\|^{2} \leq \frac{1}{N}\big(\sum_{i=1}^{N}\left\|g_{i}\left(x_{i}(k)\right)-g_{i}\left(x_{i}^*\right)\right\|\big)^{2},\nonumber \\
	&\leq \frac{1}{N}\big(\sum_{i=1}^{N} l_{3}\left\|x_{i}(k)-x_{i}^{*}\right\|\big)^{2} \leq l_{3}^{2} \sum_{i=1}^{N}\left\|x_{i}(k)-x_{i}^{*}\right\|^{2} =l_{3}^{2}\left\|\boldsymbol{x}(k)-\boldsymbol{x}^{*}\right\|^{2}. \label{zhongjian}
\end{flalign}
Substituting \eqref{zhongjian} into \eqref{xkx}, we conclude that
\begin{flalign}
		&\left\|\boldsymbol{x}(k+1)-\boldsymbol{x}(k)\right\|\leq \alpha l_{1}(1+l_{3})\left\|\boldsymbol{x}(k)-\boldsymbol{x}^{*}\right\|+\alpha l_{1}\big\|\boldsymbol{\chi}(k)-(\frac{1}{N}\mathbf{1}_{N}\mathbf{1}_{N}^{T}\otimes I_{r})\boldsymbol{\chi}(k)\big\|\nonumber\\
		&\quad+\alpha l_{3}\big\|\boldsymbol{y}(k)-(\frac{1}{N}\mathbf{1}_{N}\mathbf{1}_{N}^{T}\otimes I_{r})\boldsymbol{y}(k)\big\|.\label{xkx1}
\end{flalign}
Further, substituting \eqref{xkx1} into \eqref{chikchi}, we conclude that
\begin{flalign}
		&\big\|\boldsymbol{\chi}(k+1)-\mathbf{1}_{N}\otimes\bar{\chi}(k+1)\big\|\nonumber\\
		&\leq (\kappa\!+\!\alpha l_{1}l_{3})\big\|\boldsymbol{\chi}(k)\!-\!\mathbf{1}_{N}\otimes\bar{\chi}(k)\big\|\!+\!2\|\boldsymbol{e}_{\boldsymbol{\chi}}(k)\|\!+\!\alpha l_{1}l_{3}(1\!+\!l_{3})\|\boldsymbol{x}(k)\!-\!\boldsymbol{x}^*\|\nonumber\\
		&\quad+\alpha l_{3}^2\|\boldsymbol{y}(k)-\mathbf{1}_{N}\otimes\bar{y}(k)\|.\label{chikchi1}	
\end{flalign}
\item Invoking \eqref{ye} into $\boldsymbol{y}(k+1)-\mathbf{1}_{N}\otimes\bar{y}(k+1)$ yields that
\begin{flalign}
		&\|\boldsymbol{y}(k+1)-(\frac{1}{N}\mathbf{1}_{N}\mathbf{1}_{N}^{T}\otimes I_{r})\boldsymbol{y}(k+1)\|\leq \kappa\|\boldsymbol{y}(k)-(\frac{1}{N}\mathbf{1}_{N}\mathbf{1}_{N}^{T}\otimes I_{r})\boldsymbol{y}(k)\|\nonumber\\
		&\quad+\|\nabla_{\boldsymbol{\chi}}f(\boldsymbol{x}(k+1),\boldsymbol{\chi}(k+1))-\nabla_{\boldsymbol{\chi}}f(\boldsymbol{x}(k),\boldsymbol{\chi}(k))\|+2\|\boldsymbol{e}_{\boldsymbol{y}}(k)\|.\label{yky}	
\end{flalign}
Using \eqref{chi}, via Assumption~\ref{assumption2}, we obtain that
\begin{eqnarray}
	\begin{aligned}
&\|\nabla_{\boldsymbol{\chi}}f(\boldsymbol{x}(k+1),\boldsymbol{\chi}(k+1))-\nabla_{\boldsymbol{\chi}}f(\boldsymbol{x}(k),\boldsymbol{\chi}(k))\|\\
&\leq l_{2}\big(\|\boldsymbol{x}(k+1)-\boldsymbol{x}(k)\|+\|\boldsymbol{\chi}(k+1)-\boldsymbol{\chi}(k)\|\big),\\
&\leq l_{2}\big(\|\boldsymbol{x}(k\!+\!1)\!-\!\boldsymbol{x}(k)\|\!+\!\|\boldsymbol{L}\boldsymbol{e}_{\boldsymbol{\chi}}(k)\!+\!g(\boldsymbol{x}(k\!+\!1))\!-\!g(\boldsymbol{x}(k))\!+\!(A\otimes I_{r})\boldsymbol{\chi}(k)\!-\!\boldsymbol{\chi}(k)\|\big),\\
&\leq l_{2}\big(\|\boldsymbol{x}(k+1)-\boldsymbol{x}(k)\|+\|\boldsymbol{L}\boldsymbol{e}_{\boldsymbol{\chi}}(k)+g(\boldsymbol{x}(k+1))-g(\boldsymbol{x}(k))\\
&\quad+(A\otimes I_{r}-I_{Nr})(\boldsymbol{\chi}(k)-(\frac{1}{N}\mathbf{1}_{N}\mathbf{1}_{N}^{T}\otimes I_{r})\boldsymbol{\chi}(k))\|\big),\\
&\leq
l_{2}(1\!+\!l_{3})\|\boldsymbol{x}(k\!+\!1)\!\!-\!\boldsymbol{x}(k)\|\!\!+\!2l_{2}\|\boldsymbol{\chi}(k)\!-\!(\frac{1}{N}\mathbf{1}_{N}\mathbf{1}_{N}^{T}\!\otimes\! I_{r})\boldsymbol{\chi}(k)\|\!\!+\!2l_{2}\|\boldsymbol{e}_{\boldsymbol{\chi}}(k)\|.\label{zhongjian2}	
\end{aligned}	
\end{eqnarray}
Substituting \eqref{zhongjian2} into \eqref{yky}, one has that
\begin{flalign}
		&\|\boldsymbol{y}(k+1)\!-\!\mathbf{1}_{N}\otimes\bar{y}(k+1)\|\nonumber\\
		&\leq (\kappa\!+\!\alpha l_{2}l_{3}(1\!+\!l_{3}))\|\boldsymbol{y}(k)\!-\!\mathbf{1}_{N}\otimes\bar{y}(k)\|\!+\!\alpha l_{1}l_{2}(1\!+\!l_{3})^2\|\boldsymbol{x}(k)-\boldsymbol{x}^*\|\!+\!(\alpha l_{1}l_{2}\nonumber\\
		&\quad(1\!+\!l_{3})\!+\!2l_{2})\|\boldsymbol{\chi}(k)\!-\!\mathbf{1}_{N}\otimes\bar{\chi}(k)\|\!+\!2l_{2}\|\boldsymbol{e}_{\boldsymbol{\chi}}(k)\|\!+\!2\|\boldsymbol{e}_{\boldsymbol{y}}(k)\|.
\end{flalign}
\end{itemize}
\textbf{Step 2.} To sum up, we have bounded the $\|\boldsymbol{x}(k+1)-\boldsymbol{x}^*\|$, $\|\boldsymbol{\chi}(k+1)-\mathbf{1}_{N}\otimes\bar{\chi}(k+1)\|$ and $\|\boldsymbol{y}(k+1)-\mathbf{1}_{N}\otimes\bar{y}(k+1)\|$. Based on the definition of $\Theta(k)$, by substituting the upper bound of the norm of each coordinate, we conclude that
\begin{eqnarray}
	\Theta(k+1)\leq H(\alpha)\Theta(k)+E(k).
\end{eqnarray}
We now prove $\rho(H)<1$. Let $\lambda(\alpha)$ be the eigenvalue of $H(\alpha)$. It is obvious that $1$ is a simple eigenvalue of $H(0)$ and $\nu=[1;0;0]$ is $1$'s corresponding left and right eigenvectors. According to~\cite{Lemma1} and Lemma 5 in~\cite{lixiuxian}, there is $\frac{d\lambda(\alpha)}{d\alpha}|_{\alpha=0}=-\mu<0$, which means that the spectral radius of $H(\alpha)$ is less than $1$ for sufficiently small positive $\alpha$. Furthermore, the strongly connected graph makes the related $H(\alpha)$ is irreducible~\cite{bukeyue}. Combining with Lemmas 1-2 in~\cite{lixiuxian} can ensure that $1$ will be a simple eigenvalue of $H(\alpha)$ as $\alpha$ increases from zero to some positive value. By computing $\det(I-H(\alpha))=0$, we obtain that $$\alpha=\frac{\mu(1-\kappa)^2}{l_{3}(\mu+l_{1}+l_{2}l_{3})\big((1-\kappa)(l_{1}+l_{2}+l_{2}l_{3})+2l_{2}l_{3}\big)}.$$
Hence, when $\alpha\in\big(0,\frac{\mu(1-\kappa)^2}{l_{3}(\mu+l_{1}+l_{2}l_{3})((1-\kappa)(l_{1}+l_{2}+l_{2}l_{3})+2l_{2}l_{3})}\big)$, all eigenvalues of $H(\alpha)$ have absolute values less than $1$.

\subsection{Proof of Lemma~\ref{Lemma4}}\label{a52}
By Shur theorem, there exists an unitary matrix $U=[u_{1},u_{2},u_{3}]\in{\mathbb{R}^{3\times3}}$ such that $U^{H}H(\alpha)U$ is an upper triangular matrix as the following form
\begin{equation*}
	U^{H}H(\alpha)U=\left[
	\begin{array}{ccc}
		\lambda_{H1} & T_{12} & T_{13}  \\
	0 & \lambda_{H2} & T_{23} \\
		0 & 0 & \lambda_{H3} \\
	\end{array}
	\right],
\end{equation*}
where $T_{12}=u_{1}^{H}H(\alpha)u_{2}$, $T_{13}=u_{1}^{H}H(\alpha)u_{3}$ and $T_{23}=u_{2}^{H}H(\alpha)u_{3}$. Further define $D=\text{diag}(1,\eta,\eta^2)$ with $\eta=\epsilon/2\|H(\alpha)\|$, then it follows that
\begin{equation*}
	D^{-1}U^{H}H(\alpha)UD=\left[
	\begin{array}{ccc}
		\lambda_{H1} & T_{12}\eta & T_{13}\eta^2  \\
		0 & \lambda_{H2} & T_{23}\eta \\
		0 & 0 & \lambda_{H3} \\
	\end{array}
	\right].
\end{equation*}
It should be noted that $\max\{\|T_{12}\|,\|T_{13}\|, \|T_{23}\|\}\leq \|H(\alpha)\|$ such that
\begin{flalign} &\|D^{-1}U^{H}H(\alpha)UD\|_{\infty}\leq \rho(H(\alpha))+\epsilon,\nonumber\\
&\|D^{-1}U^{H}H^{k}(\alpha)UD\|_{\infty}\leq
\|(D^{-1}U^{H}H(\alpha)UD)^{k}\|_{\infty}\leq (\rho(H(\alpha))+\epsilon)^{k}.
\end{flalign} 
We can calculate $\|H^{k}(\alpha)\|_{\infty}$ as follows
\begin{flalign} \|H^{k}(\alpha)\|_{\infty}&=\|UDD^{-1}U^{H}H^{k}(\alpha)UDD^{-1}U^{H}\|_{\infty},\nonumber\\
&\leq \|U\|_{\infty}\|U^{H}\|_{\infty}\|D^{-1}\|_{\infty}\|D\|_{\infty}(\rho(H(\alpha))+\epsilon)^{k},\nonumber\\
&\leq 3\text{max}\{\frac{4\|H(\alpha)\|^2}{\epsilon^2},\frac{\epsilon^2}{4\|H(\alpha)\|^2}\}(\rho(H(\alpha))+\epsilon)^{k},
\end{flalign} 
which further yields that 
\begin{equation}
\|H^{k}(\alpha)\|\leq\sqrt{3}\|H^{k}(\alpha)\|_{\infty}=3\sqrt{3}\text{max}\{\frac{4\|H(\alpha)\|^2}{\epsilon^2},\frac{\epsilon^2}{4\|H(\alpha)\|^2}\}(\rho(H(\alpha))+\epsilon)^{k}.
\end{equation}

\section*{ACKNOWLEDGEMENT}
\small
This work was supported in part by the National Natural Science Foundation of China under grant 61903027, 72171171 and in art by the China Postdoctoral Science Foundation under grant 2021M702481.
\makesubmdate

\makecontacts

\end{document}